\documentclass{article}

\usepackage{bm}
\usepackage{enumerate}
\usepackage{comment}
\usepackage{amsmath,amsthm,amssymb}

\newtheorem{thm}{Theorem}[section]
\newtheorem{cor}[thm]{Corollary}
\newtheorem{prop}[thm]{Proposition}
\newtheorem{lem}[thm]{Lemma}

\def\Q{\mathbb Q}
\def\R{\mathbb R}
\def\Z{\mathbb Z}
\DeclareMathOperator{\Aut}{\operatorname{Aut}}

\DeclareMathOperator{\GL}{\operatorname{GL}}

\begin{document}

\title{On the Automorphism Group of a Binary Form Associated with Algebraic Trigonometric Quantities}


\author{Anton Mosunov}



\date{}

\maketitle

\begin{abstract}
Let $F(x, y)$ be a binary form of degree at least three and non-zero discriminant. In this article we compute the automorphism group $\Aut F$ for four families of binary forms. The first two families that we are interested in are homogenizations of minimal polynomials of $2\cos\left(\frac{2\pi}{n}\right)$ and $2\sin\left(\frac{2\pi}{n}\right)$, which we denote by $\Psi_n(x, y)$ and $\Pi_n(x, y)$, respectively. The remaining two forms that we consider are homogenizations of Chebyshev polynomials of first and second kinds, denoted $T_n(x, y)$ and $U_n(x, y)$, respectively.
\end{abstract}

\section{Introduction}
Let $F(x, y)$ denote a binary form with complex coefficients of degree $d \geq 1$. Then for a matrix $M = \left(\begin{smallmatrix}s & u\\t & v\end{smallmatrix}\right)$, with complex entries we define a new binary form $F_M(x, y)$ as follows:
$$
F_M(x, y) = F(sx + uy, tx + vy).
$$
If $K$ is a subfield of $\mathbb C$, we say that $M$ is a $K$-\emph{automorphism} of $F$ if $F_M = F$ and $M$ has entries in $K$. The set of all $K$-automorphisms of $F$ forms a group and is denoted by $\Aut_K F$. We put $\Aut F = \Aut_{\Q} F$.

In what follows, we restrict our attention to binary forms $F$ with integer coefficients, degree $d \geq 3$ and non-zero discriminant $D_F$. In this case $\Aut F$ is a finite subgroup of $\GL_2(\Q)$ \cite{stewart-xiao}. It is a fact that every finite subgroup of $\GL_2(\Q)$ is $\GL_2(\Q)$-conjugate to one of the groups listed in Table \ref{tab:finite_subgroups} \cite{newman}.

The automorphism group $\Aut F$ arises in the analysis of \emph{Thue equations}. \mbox{A Thue} equation is a Diophantine equation of the form
\begin{equation} \label{eq:thue}
F(x, y) = h,
\end{equation}
where $h$ is a fixed integer. In 1909, Thue \cite{thue} proved that this equation has only finitely many solutions in integers $x$ and $y$. A solution $(x, y)$ to such equation is called \emph{primitive} if $x$ and $y$ are coprime. 
%
By fixing a finite subgroup $G$ of $\operatorname{GL}_2(\mathbb Z)$, Stewart \cite[Section 6]{stewart1} constructed binary forms $F$ such that \mbox{$\Aut F = G$}. If we now let $M = \left(\begin{smallmatrix}s & u\\t & v\end{smallmatrix}\right)$ to be an element of $\Aut F$ and assume that $(x, y)$ is a primitive solution to (\ref{eq:thue}), then $(sx + uy, tx + vy)$ is also a primitive solution. Therefore, for infinitely many integers $h$, the equation (\ref{eq:thue}) has at least $|\Aut F|$ primitive solutions.

\begin{table}[t]
\centering
\begin{tabular}{| l | l | l | l |}
\hline
Group & Generators & Group & Generators\\
\hline\rule{0pt}{4ex}
$\bm C_1$ &
$\begin{pmatrix}1 & 0\\0 & 1\end{pmatrix}$ &
$\bm D_1$ &
$\begin{pmatrix}0 & 1\\1 & 0\end{pmatrix}$\\\rule{0pt}{5ex}
$\bm C_2$ &
$\begin{pmatrix}-1 & 0\\0 & -1\end{pmatrix}$ &
$\bm D_2$ &
$\begin{pmatrix}0 & 1\\1 & 0\end{pmatrix}, \begin{pmatrix}-1 & 0\\0 & -1\end{pmatrix}$\\\rule{0pt}{5ex}
$\bm C_3$ &
$\begin{pmatrix}0 & 1\\-1 & -1\end{pmatrix}$ &
$\bm D_3$ &
$\begin{pmatrix}0 & 1\\1 & 0\end{pmatrix}, \begin{pmatrix}0 & 1\\-1 & -1\end{pmatrix}$\\\rule{0pt}{5ex}
$\bm C_4$ &
$\begin{pmatrix}0 & 1\\-1 & 0\end{pmatrix}$ &
$\bm D_4$ &
$\begin{pmatrix}0 & 1\\1 & 0\end{pmatrix}, \begin{pmatrix}0 & 1\\-1 & 0\end{pmatrix}$\\\rule{0pt}{5ex}
$\bm C_6$ &
$\begin{pmatrix}0 & -1\\1 & 1\end{pmatrix}$ &
$\bm D_6$ &
$\begin{pmatrix}0 & 1\\1 & 0\end{pmatrix}, \begin{pmatrix}0 & 1\\-1 & 1\end{pmatrix}$\\
\hline
\end{tabular}
\caption{Representatives of equivalence classes of finite subgroups of $\GL_2(\Q)$ under conjugation.}
\label{tab:finite_subgroups}
\end{table}

In 2019, Stewart and Xiao \cite{stewart-xiao} proved that the number of integers $R_F(Z)$ of absolute value at most $Z$ which are represented by $F$ is asymptotic to $C_FZ^{2/d}$ for some positive number $C_F$. The number $C_F$ can be computed as follows. Let
$$
\{(x, y) \in \mathbb R^2 \colon |F(x, y)| \leq 1\}
$$
be the \emph{fundamental region} of $F$, and let $A_F$ denote its area. Then \mbox{$C_F = W_FA_F$}, where $W_F$ is an explicit function of $\Aut F$ described in \mbox{\cite[Theorem 1.2]{stewart-xiao}}. In \mbox{\cite[Corollary 1.3]{stewart-xiao}}, Stewart and Xiao determined $\Aut F$, $W_F$, $A_F$ and $C_F$ in the case when $F(x, y) = ax^d + by^d$ is a binomial form. Hooley \cite{hooley67, hooley00} determined the value of $A_F$ in the case when $F$ is a cubic form. In turn, Bean \cite{bean97} determined the value of $A_F$ in the case when $F$ is a quartic form. These results enabled Xiao \cite[Theorems 3.1 and 4.1]{xiao} to compute $\Aut F$, $W_F$ and $C_F$ in the case when $F$ is a cubic form or a quartic form.

In this article we compute $\Aut F$ and $W_F$, and estimate $C_F$ for four families of binary forms. Let $\Psi_n(x)$ and $\Pi_n(x)$ denote the minimal polynomials of $2\cos\left(\frac{2\pi}{n}\right)$ and $2\sin\left(\frac{2\pi}{n}\right)$, respectively. The first two families that we are interested in are $\Psi_n(x, y)$ and $\Pi_n(x, y)$, which are homogenizations of $\Psi_n(x)$ and $\Pi_n(x)$, respectively. By \cite[Lemma]{watkins-zeitlin},
\begin{equation} \label{eq:psin}
\Psi_n(x, y) = \prod\limits_{\substack{1 \leq k < \frac{n}{2}\\\gcd(k, n) = 1}}\left(x - 2\cos\left(\frac{2\pi k}{n}\right)y\right).
\end{equation}
Further, since $\sin\left(\frac{2\pi}{n}\right) = \cos\left(\frac{2\pi(n - 4)}{4n}\right)$, we see that $\sin\left(\frac{2\pi}{n}\right)$ is an algebraic conjugate of $\cos\left(\frac{2\pi}{c(n)}\right)$, where $c(n)$ is the denominator of $\frac{n - 4}{4n}$ (in lowest terms). Consequently,
\begin{equation} \label{eq:Pi}
\Pi_n(x, y) = \Psi_{c(n)}(x, y).
\end{equation}
The formula for $c(n)$ can be found in \cite[Corollary 1.2]{mosunov}.

Next, let $T_n(x)$ and $U_n(x)$ denote Chebyshev polynomials of first and second kinds, respectively. The other two families that we are interested in are $T_n(x, y)$ and $U_n(x, y)$, which are homogenizations of $T_n(x)$ and $U_n(x)$, respectively. It is known \cite{mason-handscomb} that
$$
T_n(x, y) = 2^{n - 1}\prod\limits_{k = 0}^{n-1}\left(x - \cos\left(\frac{(2k + 1)\pi}{2n}\right)y\right)
$$
and
$$
U_n(x, y) = 2^n\prod\limits_{k = 1}^n\left(x - \cos\left(\frac{k\pi}{n + 1}\right)y\right).
$$

Let $\varphi(n)$ denote the Euler's totient function. Define $\Aut |F|$ as the group of all $2 \times 2$ matrices $M$, with rational entries, such that $F_M = F$ or \mbox{$F_M = -F$}. Note that $\Aut |F|$ contains finitely many elements, and that $\Aut F$ is a normal subgroup of $\Aut |F|$ of index at most $2$. Our first result is stated in \mbox{Theorem \ref{thm:cos}}.

\begin{table}[t]
\centering
\begin{tabular}{| c | c | l | c | l |}
\hline
& \multicolumn{2}{c|}{$\Aut \Psi_n$} & \multicolumn{2}{c|}{$\Aut |\Psi_n|$}\\
\hline
$n$ & $\operatorname{Rep} \Psi_n$ & Generators & $\operatorname{Rep} |\Psi_n|$ & Generators\\
\hline
$7, 18$ & $\bm C_3$ & $\begin{pmatrix}-1 & -1\\1 & 0\end{pmatrix}$ & $\bm D_3$ & $\begin{pmatrix}-1 & -1\\1 & 0\end{pmatrix}, \begin{pmatrix}-1 & 0\\0 & -1\end{pmatrix}$\\
$9, 14$ & $\bm C_3$ & $\begin{pmatrix}-1 & 1\\-1 & 0\end{pmatrix}$ & $\bm D_3$ & $\begin{pmatrix}-1 & 1\\-1 & 0\end{pmatrix}, \begin{pmatrix}-1 & 0\\0 & -1\end{pmatrix}$\\
$15$ & $\bm C_4$ & $\begin{pmatrix}-1 & 2\\-1 & 1\end{pmatrix}$ & $\bm C_4$ & $\begin{pmatrix}-1 & 2\\-1 & 1\end{pmatrix}$\\
$24$ & $\bm D_4$ & $\begin{pmatrix}0 & 1\\1 & 0\end{pmatrix}, \begin{pmatrix}0 & 1\\-1 & 0\end{pmatrix}$ & $\bm D_4$ & $\begin{pmatrix}0 & 1\\1 & 0\end{pmatrix}, \begin{pmatrix}0 & 1\\-1 & 0\end{pmatrix}$\\
$30$ & $\bm C_4$ & $\begin{pmatrix}1 & 2\\-1 & -1\end{pmatrix}$ & $\bm C_4$ & $\begin{pmatrix}1 & 2\\-1 & -1\end{pmatrix}$\\
\hline
\end{tabular}
\caption{$\Aut \Psi_n$ and $\Aut |\Psi_n|$ for $n \in \{7, 9, 14, 15, 18, 24, 30\}$. Here $\operatorname{Rep} \Psi_n$ and $\operatorname{Rep}|\Psi_n|$ denote representatives of the equivalence classes of $\Aut \Psi_n$ and $\Aut |\Psi_n|$, respectively, under $\GL_2(\mathbb Q)$ conjugation.}
\label{tab:automorphisms_cos}
\end{table}

\begin{thm} \label{thm:cos}
Let $n$ be a positive integer such that $n \notin \{1, 2, 3, 4, 5, 6, 8, 10, 12\}$ and let $d = \varphi(n)/2$, so that $\deg \Psi_n = d$ and $d \geq 3$.

\begin{enumerate}[1.]
\item If $d \geq 5$ is odd, then $\Aut \Psi_n = \{I\} \cong \bm C_1$ and $\Aut |\Psi_n| = \{\pm I\} \cong \bm C_2$, \mbox{where $I$} denotes the $2 \times 2$ identity matrix.

\item If $d \geq 6$ is even and $n \not \equiv 0 \pmod 4$, then $\Aut \Psi_n = \Aut |\Psi_n| = \{\pm I\} \cong \bm C_2$.

\item If $n \equiv 0 \pmod 4$ and $n \neq 24$, then
$$
\Aut \Psi_n = \Aut |\Psi_n| = \left\langle
\begin{pmatrix}
-1 & 0\\
0 & 1
\end{pmatrix},
\begin{pmatrix}
1 & 0\\
0 & -1
\end{pmatrix}
\right\rangle
\cong \bm D_2.
$$

\item If $n \in \{7, 9, 14, 15, 18, 24, 30\}$, then $\Aut \Psi_n$ and $\Aut |\Psi_n|$ are as in Table \ref{tab:automorphisms_cos}.
\end{enumerate}
\end{thm}

Since $\deg \Psi_n \in \{3, 4\}$ if and only if $n \in \{7, 9, 14, 15, 16, 18, 20, 24, 30\}$, we see that Theorem \ref{thm:cos} covers all possible cases. The proof relies on the careful analysis of roots of $\Psi_n(x)$ given in (\ref{eq:psin}). For example, in \mbox{Lemma \ref{lem:stuv}} we use the fact that, for every odd positive integer $n$, if $\alpha_i = 2\cos\left(\frac{2\pi i}{n}\right)$ is a root of $\Psi_n(x)$, then so is $\alpha_{2i} = \alpha_i^2 - 2$, while in Lemma \ref{lem:reciprocal} we use the fact that no root of $\Psi_n(x)$ exceeds $2$ in absolute value.

From Theorem \ref{thm:cos} we deduce the following.

\begin{table}[t]
\centering
\begin{tabular}{| c | c | l | c | l |}
\hline
& \multicolumn{2}{c|}{$\Aut \Pi_n$} & \multicolumn{2}{c|}{$\Aut |\Pi_n|$}\\
\hline
$n$ & $\operatorname{Rep} \Pi_n$ & Generators & $\operatorname{Rep} |\Pi_n|$ & Generators\\
\hline
$28, 36$ & $\bm C_3$ & $\begin{pmatrix}-1 & 1\\-1 & 0\end{pmatrix}$ & $\bm D_3$ & $\begin{pmatrix}-1 & 1\\-1 & 0\end{pmatrix}, \begin{pmatrix}-1 & 0\\0 & -1\end{pmatrix}$\\
$60$ & $\bm C_4$ & $\begin{pmatrix}1 & 2\\-1 & -1\end{pmatrix}$ & $\bm C_4$ & $\begin{pmatrix}1 & 2\\-1 & -1\end{pmatrix}$\\
$24$ & $\bm D_4$ & $\begin{pmatrix}0 & 1\\1 & 0\end{pmatrix}, \begin{pmatrix}0 & 1\\-1 & 0\end{pmatrix}$ & $\bm D_4$ & $\begin{pmatrix}0 & 1\\1 & 0\end{pmatrix}, \begin{pmatrix}0 & 1\\-1 & 0\end{pmatrix}$\\
\hline
\end{tabular}
\caption{$\Aut \Pi_n$ and $\Aut |\Pi_n|$ for $n \in \{24, 28, 36, 60\}$. Here $\operatorname{Rep} \Pi_n$ and $\operatorname{Rep}|\Pi_n|$ denote representatives of the equivalence classes of $\Aut \Pi_n$ and $\Aut |\Pi_n|$, respectively, under $\GL_2(\mathbb Q)$ conjugation.}
\label{tab:automorphisms_sin}
\end{table}

\begin{cor} \label{cor:sin}
Let $n$ be a positive integer such that $n \notin \{1, 2, 3, 4, 6, 8, 12, 20\}$. Let
\begin{equation} \label{eq:d-pi}
d = \begin{cases}
\varphi(n) & \text{if $\gcd(n, 8) < 4$,}\\
\varphi(n)/4 & \text{if $\gcd(n, 8) = 4$ and $n \neq 4$,}\\
\varphi(n)/2 & \text{if $\gcd(n, 8) > 4$,}
\end{cases}
\end{equation}
so that $\deg \Pi_n = d$ and $d \geq 3$ \cite[III.4]{niven}.

\begin{enumerate}[1.]
\item If $n \equiv 4 \pmod 8$ and $d \geq 5$ is odd, then $\Aut \Pi_n = \{I\} \cong \bm C_1$ and $\Aut |\Pi_n| = \{\pm I\} \cong \bm C_2$.

\item If $n \equiv 4 \pmod 8$ and $d \geq 6$ is even, then $\Aut \Pi_n = \Aut |\Pi_n| = \{\pm I\} \cong \bm C_2$.

\item If $n \not \equiv 4 \pmod 8$ and $n \neq 24$, then
$$
\Aut \Pi_n = \left\langle
\begin{pmatrix}
-1 & 0\\
0 & 1
\end{pmatrix},
\begin{pmatrix}
1 & 0\\
0 & -1
\end{pmatrix}
\right\rangle
\cong \bm D_2.
$$

\item If $n \in \{24, 28, 36, 60\}$, then $\Aut \Pi_n$ and $\Aut |\Pi_n|$ are as in Table \ref{tab:automorphisms_sin}.
\end{enumerate}
\end{cor}

\begin{proof}
This is a direct consequence of (\ref{eq:Pi}) and Theorem \ref{thm:cos}.
\end{proof}
Since $\deg \Pi_n \in \{3, 4\}$ if and only if $n \in \{5,10,16,24,28,36,60\}$, we see that Corollary \ref{cor:sin} covers all possible cases.

Our second result is stated in \mbox{Theorem \ref{thm:chebyshev}}.

\begin{thm} \label{thm:chebyshev}
For an integer $n \geq 3$, let $T_n(x, y)$ and $U_n(x, y)$ denote the homogenizations of the $n$-th Chebyshev polynomials of first and second kinds, respectively.

\begin{enumerate}[1.]
\item If $n$ is odd, then
$$
\Aut T_n = \left\langle
\begin{pmatrix}
-1 & 0\\
0 & 1
\end{pmatrix}
\right\rangle
\cong \bm C_2,
\quad
\Aut |T_n| = \left\langle
\begin{pmatrix}
-1 & 0\\
0 & 1
\end{pmatrix},
\begin{pmatrix}
1 & 0\\
0 & -1
\end{pmatrix}
\right\rangle
\cong \bm D_2,
$$
$$
\Aut U_n = \left\langle
\begin{pmatrix}
-1 & 0\\
0 & 1
\end{pmatrix}
\right\rangle
\cong \bm C_2,
\quad
\Aut |U_n| = \left\langle
\begin{pmatrix}
-1 & 0\\
0 & 1
\end{pmatrix},
\begin{pmatrix}
1 & 0\\
0 & -1
\end{pmatrix}
\right\rangle
\cong \bm D_2.
$$

\item If $n$ is even, then
$$
\Aut T_n = \Aut |T_n| = \left\langle
\begin{pmatrix}
-1 & 0\\
0 & 1
\end{pmatrix},
\begin{pmatrix}
1 & 0\\
0 & -1
\end{pmatrix}
\right\rangle
\cong \bm D_2,
$$
$$
\Aut U_n = \Aut |U_n| = \left\langle
\begin{pmatrix}
-1 & 0\\
0 & 1
\end{pmatrix},
\begin{pmatrix}
1 & 0\\
0 & -1
\end{pmatrix}
\right\rangle
\cong \bm D_2.
$$
\end{enumerate}
\end{thm}

The article is organized as follows. In Section \ref{sec:stewart-xiao} we use Theorem \ref{thm:cos}, Corollary \ref{cor:sin} and Theorem \ref{thm:chebyshev} to derive estimates for the quantities $C_{\Psi_n}$, $C_{\Pi_n}$, $C_{T_n}$ and $C_{U_n}$. In Section \ref{sec:theory} we prove seven preliminary lemmas. Readers may skip the proofs in Section \ref{sec:theory} and refer only to the results when reading proofs of Theorems \ref{thm:cos} and \ref{thm:chebyshev}, which are outlined in Sections \ref{sec:cos} and \ref{sec:chebyshev}, respectively.

\section{Computation of $C_F$} \label{sec:stewart-xiao}

\begin{table}[t]
\centering
\begin{tabular}{c | c | c | c | c | c | c}
$n$ & $W_{\Psi_n}$ & $A_{\Psi_n}$ & $C_{\Psi_n}$ & $W_{\Pi_n}$ & $A_{\Pi_n}$ & $C_{\Pi_n}$\\
\hline
5 & --- & $\infty$ & --- & 1/4 & 5.78302 & 1.44575\\
7 & 1/3 & 8.31171 & 2.77057 & 1/4 & 5.38644 & 1.34661\\
9 & 1/3 & 7.64379 & 2.54793 & 1/4 & 5.63543 & 1.40886\\
10 & --- & $\infty$ & --- & 1/4 & 5.78302 & 1.44575\\
11 & 1 & 6.12984 & 6.12984 & 1/4 & 5.27188 & 1.31797\\
13 & 1/2 & 5.8883 & 2.94415 & 1/4 & 5.26356 & 1.31589\\
14 & 1/3 & 8.31171 & 2.77057 & 1/4 & 5.38644 & 1.34661\\
15 & 1/4 & 6.31617 & 1.57904 & 1/4 & 5.84408 & 1.46102\\
16 & 1/4 & 6.08123 & 1.52031 & 1/4 & 6.08123 & 1.52031\\
17 & 1/2 & 5.66529 & 2.83265 & 1/4 & 5.26355 & 1.31589
\end{tabular}
\caption{Invariants associated with $\Psi_n$ and $\Pi_n$ for $n \in \{5, 7, 9, 10, 11, \ldots, 17\}$.}
\label{tab:table2}
\end{table}

In this section we estimate the quantity $C_F = W_FA_F$ for binary forms $\Psi_n$, $\Pi_n$, $T_n$ and $U_n$. For small values of $n$ the constants $C_{\Psi_n}$, $C_{\Pi_n}$, $C_{T_n}$ and $C_{U_n}$, along with other invariants, can be found in Tables \ref{tab:table2} and \ref{tab:table3}. The quantities $A_{\Psi_n}$, $A_{\Pi_n}$, $A_{T_n}$ and $A_{U_n}$ were estimated (but not computed) by the author in \cite{mosunov}. In particular, using lower- and upper-bounds on $A_{\Psi_n}$, $A_{\Pi_n}$, $A_{T_n}$ and $A_{U_n}$ established in \cite{mosunov}, one can prove that
\begin{equation} \label{eq:limit-psi}
\lim\limits_{n \rightarrow \infty} A_{\Psi_n} = \lim\limits_{n \rightarrow \infty} A_{\Pi_n} = \frac{16}{3}
\end{equation}
and
\begin{equation} \label{eq:limit-chebyshev}
\lim\limits_{n \rightarrow \infty} A_{T_n} = \lim\limits_{n \rightarrow \infty} A_{U_n} = \frac{8}{3}.
\end{equation}

It remains to compute the rational numbers $W_{\Psi_n}$, $W_{\Pi_n}$, $W_{T_n}$ and $W_{U_n}$. To do so, we use the formula provided in \cite[Theorem 1.2]{stewart-xiao}. Let $\Lambda$ be the sublattice of $\mathbb Z^2$ consisting of $(u, v)$ in $\mathbb Z^2$ for which $A\left(\begin{smallmatrix}u\\v\end{smallmatrix}\right)$ is in $\mathbb Z^2$ for all $A$ in $\Aut F$. Put $m = d(\Lambda)$, where $d(\Lambda)$ is the determinant of $\Lambda$. By \mbox{\cite[Theorem 1.2]{stewart-xiao}}, the value of $W_F$ depends on $m$ when $\Aut F$ is isomorphic to $\bf C_3$, $\bf C_4$, $\bf C_6$, $\bf D_1$ or $\bf D_2$. When $\Aut F$ is isomorphic to $\bf D_3$, $\bf D_4$ or $\bf D_6$, in addition to being dependent on $m$, the value $W_F$ depends on the quantities $m_i = d(\Lambda_i)$. The lattices $\Lambda_i$ are generated in a similar fashion as $\Lambda$ by certain \emph{subgroups} $G_i$ of $\Aut F$, whose order is \mbox{either $2$ or $3$}. When $\Aut F$ is isomorphic to $\bf C_1$ or $\bf C_2$, the value $W_F$ is equal to $1$ or $1/2$, respectively. In the special case when $\operatorname{Aut} F$ is a subgroup of $\GL_2(\mathbb Z)$, we have $m = 1$ and $m_i = 1$, and so the formula for $W_F$ becomes especially simple:
\begin{equation} \label{eq:WF-simple-formula}
W_F = \frac{1}{|\operatorname{Aut} F|}.
\end{equation}
Notice that all automorphism groups in Theorem \ref{thm:cos}, Corollary \ref{cor:sin} and Theorem \ref{thm:chebyshev} are subgroups of $\operatorname{GL}_2(\mathbb Z)$, so the above formula applies. For this reason we omit the calculations and directly state our results. Notice that in the following propositions we do not provide explicit formulas for $C_{\Psi_n}$, $C_{\Pi_n}$, $C_{T_n}$ and $C_{U_n}$, because, as it was mentioned above, the quantities $A_{\Psi_n}$, $A_{\Pi_n}$, $A_{T_n}$ and $A_{U_n}$ were \emph{estimated}, but not \emph{computed}.

\begin{prop} \label{prop:W-psi}
Let $n$ be a positive integer such that $n \notin \{1, 2, 3, 4, 5, 6, 8, 10, 12\}$ and let $d = \varphi(n)/2$, so that $\deg \Psi_n = d$ and $d \geq 3$. Then
$$
W_{\Psi_n} = \begin{cases}
1 & \text{if $d \geq 5$ is odd,}\\
1/2 & \text{if $d \geq 6$ is even and $n \not \equiv 0 \!\!\! \pmod 4$,}\\
1/3 & \text{if $n \in \{7, 9, 14, 18\}$,}\\
1/4 & \text{if $n \in \{15, 30\}$,}\\
1/4 & \text{if $n \equiv 0 \!\!\! \pmod{4}$ and $n \neq 24$,}\\
1/8 & \text{if $n = 24$.}
\end{cases}
$$
Consequently, if we let $R_{\Psi_n}(Z)$ denote the number of integers of absolute value at most $Z$ which are represented by $\Psi_n$, then
$$
R_{\Psi_n}(Z) \sim W_{\Psi_n}A_{\Psi_n}Z^{2/d},
$$
with lower- and upper-bound on $A_{\Psi_n}$ given in \cite[Theorem 1.1]{mosunov}.
\end{prop}

\begin{proof}
The formula for $W_{\Psi_n}$ is a direct consequence of Theorem \ref{thm:cos} and (\ref{eq:WF-simple-formula}). The asymptotic formula for $R_{\Psi_n}(Z)$ follows from \cite[Theorem 1.1]{stewart-xiao}.
\end{proof}

Combining Proposition \ref{prop:W-psi} with $C_F = W_FA_F$ and (\ref{eq:limit-psi}), we find that
$$
\lim\limits_{k \rightarrow \infty}C_{\Psi_{4k}} = \frac{4}{3}.
$$
Further, note that $\deg \Psi_n \geq 3$ is odd if and only if $n = p^j$ or $n = 2p^j$, where \mbox{$p \equiv 3$} \mbox{(mod $4$)} is prime and $j$ is a positive integer. Let $\mathcal S = \{3, 6, 7, 9, 11, 14, 18, \ldots\}$ denote the set of all such integers. Then
$$
\lim\limits_{\substack{n \rightarrow \infty\\n \in \mathcal S}} C_{\Psi_n} = \frac{16}{3} \quad \text{and} \quad \lim\limits_{\substack{n \rightarrow \infty\\ n \notin \mathcal S,\ 4 \nmid n}} C_{\Psi_n} = \frac{8}{3}.
$$

\begin{table}[t]
\centering
\begin{tabular}{c | c | c | c | c | c | c}
$n$ & $W_{T_n}$ & $A_{T_n}$ & $C_{T_n}$ & $W_{U_n}$ & $A_{U_n}$ & $C_{U_n}$\\
\hline
$3$ & 1/2 & 5.78286 & 2.89143 & 1/2 & 4.46217 & 2.23086\\
$4$ & 1/4 & 4.30008 & 1.07502 & 1/4 & 3.50332 & 0.87583\\
$5$ & 1/2 & 3.78568 & 1.89284 & 1/2 & 3.19719 & 1.59859\\
$6$ & 1/4 & 3.52082 & 0.880205 & 1/4 & 3.04985 & 0.762463\\
$7$ & 1/2 & 3.35841 & 1.6792 & 1/2 & 2.96434 & 1.48217\\
$8$ & 1/4 & 3.24832 & 0.812081 & 1/4 & 2.90894 & 0.727235\\
$9$ & 1/2 & 3.16867 & 1.58434 & 1/2 & 2.87035 & 1.43517\\
$10$ & 1/4 & 3.10831 & 0.777077 & 1/4 & 2.84203 & 0.710508\\
$11$ & 1/2 & 3.06096 & 1.53048 & 1/2 & 2.82042 & 1.41021\\
$12$ & 1/4 & 3.02282 & 0.755705 & 1/4 & 2.80343 & 0.700857
\end{tabular}
\caption{Invariants associated with $T_n$ and $U_n$ for $n \in \{3, 4, \ldots, 12\}$.}
\label{tab:table3}
\end{table}

\begin{prop} \label{prop:W-pi}
Let $n$ be a positive integer such that $n \notin \{1, 2, 3, 4, 6, 8, 12, 20\}$. Let $d$ be as in (\ref{eq:d-pi}), so that $\deg \Pi_n = d$. Then
$$
W_{\Pi_n} = \begin{cases}
1 & \text{if $n \equiv 4 \!\!\!\pmod 8$ and $d \geq 5$ is odd,}\\
1/2 & \text{if $n \equiv 4 \!\!\!\pmod 8$ and $d \geq 6$ is even,}\\
1/3 & \text{if $n \in \{28, 36\}$,}\\
1/4 & \text{if $n = 60$,}\\
1/4 & \text{if $n \not \equiv 4 \!\!\!\pmod 8$ and $n \neq 24$,}\\
1/8 & \text{if $n = 24$.}
\end{cases}
$$
Consequently, if we let $R_{\Pi_n}(Z)$ denote the number of integers of absolute value at most $Z$ which are represented by $\Pi_n$, then
$$
R_{\Pi_n}(Z) \sim W_{\Psi_n}A_{\Psi_n}Z^{2/d},
$$
with lower- and upper-bound on $A_{\Pi_n} = A_{\Psi_{c(n)}}$ given in \cite[Theorem 1.1]{mosunov}.
\end{prop}

\begin{proof}
The formula for $W_{\Pi_n}$ is a direct consequence of Corollary \ref{cor:sin} and (\ref{eq:WF-simple-formula}). The asymptotic formula for $R_{\Pi_n}(Z)$ follows from \cite[Theorem 1.1]{stewart-xiao}.
\end{proof}

Combining Proposition \ref{prop:W-pi} with $C_F = W_FA_F$ and (\ref{eq:limit-psi}), we find that
$$
\lim\limits_{\substack{n \rightarrow \infty\\ n \not \equiv 4 \!\!\!\!\!\pmod 8}}C_{\Pi_n} = \frac{4}{3}.
$$
Further, note that $\deg \Pi_n \geq 3$ is odd if and only if $n = 4p^j$, where \mbox{$p \equiv 3$} \mbox{(mod $4$)} is prime and $j$ is a positive integer. Let $\mathcal T = \{12, 28, 36, 44, 76, 92, 108, \ldots\}$ denote the set of all such integers. Then
$$
\lim\limits_{\substack{n \rightarrow \infty\\n \in \mathcal T}} C_{\Pi_n} = \frac{16}{3} \quad \text{and} \quad \lim\limits_{\substack{n \rightarrow \infty\\\substack{n \notin \mathcal T,\ 8 \mid (n - 4)}}} C_{\Pi_{n}} = \frac{8}{3}.
$$

\begin{prop} \label{prop:W-chebyshev}
Let $n$ be an integer such that $n \geq 3$. Then
$$
W_{T_n} = W_{U_n} =
\begin{cases}
1/2 & \text{if $n$ is odd,}\\
1/4 & \text{if $n$ is even.}
\end{cases}
$$
Consequently, if we let $R_{T_n}(Z)$ and $R_{U_n}(Z)$ denote the number of integers of absolute value at most $Z$ which are represented by $T_n$ and $U_n$, respectively, then
$$
R_{T_n}(Z) \sim W_{T_n}A_{T_n}Z^{2/n},
$$
$$
R_{U_n}(Z) \sim W_{U_n}A_{U_n}Z^{2/n},
$$
with lower- and upper-bounds on $A_{T_n}$ and $A_{U_n}$ given in \cite[Theorem 1.3]{mosunov} and \cite[Theorem 1.4]{mosunov}, respectively.
\end{prop}

\begin{proof}
The formulas for $W_{T_n}$ and $W_{U_n}$ can be established with Theorem \ref{thm:chebyshev} and (\ref{eq:WF-simple-formula}). The asymptotic formulas for $R_{T_n}(Z)$ and $R_{U_n}(Z)$ follow from \cite[Theorem 1.1]{stewart-xiao}.
\end{proof}

Combining Proposition \ref{prop:W-chebyshev} with $C_F = W_FA_F$ and (\ref{eq:limit-chebyshev}), we find that
$$
\lim\limits_{k \rightarrow \infty}C_{T_{2k + 1}} = \lim\limits_{k \rightarrow \infty}C_{U_{2k + 1}} = \frac{4}{3}
$$
and
$$
\lim\limits_{k \rightarrow \infty}C_{T_{2k}} = \lim\limits_{k \rightarrow \infty}C_{U_{2k}} = \frac{2}{3}.
$$

\section{Preliminary Lemmas} \label{sec:theory}

In this section we summarize some facts that will become useful to us when proving Theorems \ref{thm:cos} and \ref{thm:chebyshev}.

\begin{lem} \label{lem:galois}
Let $n$ be a positive integer. The Galois group of the field $\mathbb Q\left(2\cos\left(\frac{2\pi}{n}\right)\right)$ is Abelian and it consists of field automorphisms $\sigma_k$ defined by $\sigma_k\left(2\cos\left(\frac{2\pi}{n}\right)\right) = 2\cos\left(\frac{2\pi k}{n}\right)$, where $k$ is an integer coprime to $n$.
\end{lem}

\begin{proof}
Let $\zeta_n = e^{\frac{2\pi i}{n}}$. By \cite[Theorem 14.5.26]{dummit-foote}, the Galois group of the cyclotomic field $\mathbb Q\left(\zeta_n\right)$ is Abelian, because it is isomorphic to $(\mathbb Z/n\mathbb Z)^\times$, the multiplicative group of invertible elements in $\mathbb Z/n\mathbb Z$. Since $2\cos\left(\frac{2\pi}{n}\right) = \zeta_n + \zeta_n^{-1}$, we see that $\mathbb Q(\zeta_n + \zeta_n^{-1})$ is a subfield of $\mathbb Q(\zeta_n)$, so the Galois group of $\mathbb Q(\zeta_n + \zeta_n^{-1})$ is also Abelian. Furthermore, the Galois group of $\mathbb Q(\zeta_n)$ consists of field automorphisms $\tau_k$ defined by $\tau_k(\zeta_n) = \zeta_n^k$, where $k$ is an integer coprime to $n$. Restricting the field automorphism $\tau_k$ to the field $\mathbb Q(\zeta_n + \zeta_n^{-1})$ gives us the field automorphism $\sigma_k$ defined by $\sigma_k(\zeta_n + \zeta_n^{-1}) = \zeta_n^k + \zeta_n^{-k}$.
\end{proof}

\begin{lem} \label{lem:odd_and_even}
Let $n \geq 3$ be an integer and let $d = \varphi(n)/2$, so that $\deg \Psi_n = d$.

\begin{enumerate}[1.]
\item If $n \equiv 0 \pmod 4$, then $\Psi_n(x) = g(x^2)$, where $g(x)$ is the minimal polynomial of $2 + 2\cos\left(\frac{4\pi}{n}\right)$.

\item If $n$ is odd, then $-2\cos\left(\frac{2\pi}{n}\right)$ is a conjugate of $2\cos\left(\frac{\pi}{n}\right)$. Consequently,
$$
\Psi_n(x) = (-1)^d\Psi_{2n}(-x).
$$
\end{enumerate}
\end{lem}

\begin{proof}
\hskip1pt

\begin{enumerate}[1.]
\item Suppose that $n \equiv 0 \pmod 4$. Recall that $2\cos^2(x) = 1 + \cos(2x)$ for any $x \in \R$. Therefore,
$$
4\cos^2\left(\frac{2\pi}{n}\right) = 2\left(1 + \cos\left(\frac{4\pi}{n}\right)\right) = 2 + 2\cos\left(\frac{2\pi}{n/2}\right).
$$
Let $g(x)$ denote the minimal polynomial of $2 + 2\cos\left(\frac{2\pi}{(n/2)}\right)$. Note that $\deg g(x) = \varphi(n/2)/2$ and
$$
g\left(4\cos^2\left(\frac{2\pi}{n}\right)\right) = 0.
$$
Since for any positive integer $n$ divisible by $4$ it is the case that $\varphi(n)/2 = \varphi(n/2)$, we have
$$
\deg \Psi_n(x) = \frac{\varphi(n)}{2} = 2\cdot \frac{\varphi(n/2)}{2} = 2\deg g(x) = \deg g(x^2).
$$
Since the polynomials $g(x^2)$ and $\Psi_n(x)$ have equal degrees, both vanish at $2\cos\left(\frac{2\pi}{n}\right)$, and the leading coefficient of $g(x^2)$ is positive, we conclude that $\Psi_n(x) = g(x^2)$.

\item Suppose that $n$ is odd. Note that
$$
-2\cos\left(\frac{2\pi}{n}\right) = 2\cos\left(\pi + \frac{2\pi}{n}\right) = 2\cos\left(\frac{2\pi(n + 2)}{2n}\right).
$$
Since $\gcd(2n, n + 2) = 1$, we see that $-2\cos\left(\frac{2\pi}{n}\right)$ is a conjugate of $2\cos\left(\frac{\pi}{n}\right)$. Thus \mbox{$\Psi_{2n}\left(-2\cos\left(\frac{2\pi}{n}\right)\right) = 0$.} But then $2\cos\left(\frac{2\pi}{n}\right)$ is a root of $(-1)^d\Psi_{2n}(-x)$, and since the leading coefficient of this polynomial is positive, it must be equal to the minimal polynomial of $2\cos\left(\frac{2\pi}{n}\right)$.
\end{enumerate}
\end{proof}

\begin{lem} \label{lem:automorphism_of_adjacent_form}
Let $F(x, y) \in \Z[x, y]$ be a binary form. Let $r$ be a non-zero rational number and let $S \in \operatorname{GL}_2(\mathbb Q)$. Then
$$
\Aut rF_S = S^{-1}(\Aut F)S \quad \text{and} \quad \Aut |rF_S| = S^{-1}\left(\Aut |F|\right)S.
$$
\end{lem}

\begin{proof}
We see that $F_M = F$ if and only if
$$
F_S = (F_M)_S = F_{MS} = \left((F_S)_{S^{-1}}\right)_{MS} = (F_S)_{S^{-1}MS}.
$$
We conclude that $M \in \Aut F$ if and only if $S^{-1}MS \in \Aut F_S$. This means that $\Aut F_S = S^{-1}\left(\Aut F\right)S$. Since $\Aut rF_S = \Aut F_S$, the result follows. The equality $\Aut |rF_S| = S^{-1}\left(\Aut |F|\right)S$ can be proved analogously.
\end{proof}

\begin{lem} \label{lem:options}
Let
$$
D_2 = 
\left\langle
\begin{pmatrix}-1 & 0\\0 & 1\end{pmatrix},
\begin{pmatrix}1 & 0\\0 & -1\end{pmatrix}
\right\rangle.
$$
Every finite subgroup of $\GL_2(\Q)$ that properly contains $D_2$ is either of the form
$$
\left\langle
\begin{pmatrix}
-1 & 0\\
0 & 1
\end{pmatrix},
\begin{pmatrix}
0 & t\\
-1/t & 0
\end{pmatrix}
\right\rangle
$$
or of the form
$$
\left\langle
\begin{pmatrix}
-1 & 0\\
0 & 1
\end{pmatrix},
\begin{pmatrix}
1/2 & t/2\\
-3/(2t) & 1/2
\end{pmatrix}
\right\rangle
$$
for some non-zero $t \in \Q$.
\end{lem}

\begin{proof}
Let $G = \GL_2(\mathbb Q)$ and let $H$ be a finite subgroup of $G$ that properly contains $D_2$. According to the classification of finite subgroups of $G$ given in Table \ref{tab:finite_subgroups}, every finite subgroup of $G$ that contains a group isomorphic to $\bm D_2$ and has more than $4$ elements is $G$-conjugate to either $\bm D_4$ or $\bm D_6$. We consider these two cases separately.

\begin{enumerate}[1.]
\item Suppose that $H$ is $G$-conjugate to $\bm D_4$. That is, there exists some matrix $A \in G$ such that $H = A\bm D_4A^{-1}$. Since $D_2 \subsetneq H$, we also have
$$
D_2 = ANA^{-1}
$$
for some subgroup $N$ of $\bm D_4$ that is isomorphic to $\bm D_2$. Note that $\bm D_4$ contains exactly two subgroups isomorphic to $\bm D_2$, namely $\bm D_2$ itself and $D_2$. Thus we consider two separate cases, i.e., \mbox{$N = \bm D_2$} and \mbox{$N = D_2$}.
\begin{enumerate}
\item Suppose that $D_2 = A\bm D_2A^{-1}$. A straightforward calculation shows that every matrix $A \in G$ such that $D_2 = A\bm D_2A^{-1}$ must be of the form
$$
\begin{pmatrix}
a & -a\\b & b
\end{pmatrix}
\,\,\, \textnormal{or} \,\,\,
\begin{pmatrix}
a & a\\b & -b
\end{pmatrix}
$$
for some non-zero $a, b \in \mathbb Q$. Independently of the form of $A$, we have
\begin{align*}
H & = 
\left\langle
A
\begin{pmatrix}
0 & 1\\1 & 0
\end{pmatrix}
A^{-1},
A
\begin{pmatrix}
0 & 1\\-1 & 0
\end{pmatrix}
A^{-1}
\right\rangle\\
& = \left\langle
\begin{pmatrix}
-1 & 0\\0 & 1
\end{pmatrix},
\begin{pmatrix}
0 & a/b\\
-b/a & 0
\end{pmatrix}
\right\rangle
\end{align*}
Upon setting $t = a/b$, the result follows.

\item Suppose that $D_2 = AD_2A^{-1}$. A straightforward calculation shows that every matrix $A \in G$ such that $D_2 = AD_2A^{-1}$ must be of the form
$$
\begin{pmatrix}
a & 0\\
0 & b
\end{pmatrix}
\,\,\,\textrm{or}\,\,\,
\begin{pmatrix}
0 & a\\
b & 0
\end{pmatrix}
$$
for some non-zero $a, b \in \Q$. Consequently,
\begin{align*}
H
& =
\left\langle
\begin{pmatrix}
0 & a/b\\b/a & 0
\end{pmatrix},
\begin{pmatrix}
0 & a/b\\-b/a & 0
\end{pmatrix}
\right\rangle\\
& =
\left\langle
\begin{pmatrix}
-1 & 0\\0 & 1
\end{pmatrix},
\begin{pmatrix}
0 & a/b\\-b/a & 0
\end{pmatrix}
\right\rangle
\end{align*}
Upon setting $t = a/b$, the result follows.
\end{enumerate}

\item Suppose that $H$ is $G$-conjugate to $\bm D_6$. That is, there exists some matrix $A \in G$ such that $H = A\bm D_6A^{-1}$. Since $D_2 \subsetneq D_6$, we also have
$$
D_2 = ANA^{-1}
$$
for some subgroup $N$ of $\bm D_6$ that is isomorphic to $\bm D_2$. Note that $\bm D_6$ contains exactly three subgroups isomorphic to $\bm D_2$, namely $\bm D_2$ itself,
$$
D_2^{(1)} =
\left\langle
\begin{pmatrix}
-1 & 0\\0 & -1
\end{pmatrix},
\begin{pmatrix}
1 & 0\\1 & -1
\end{pmatrix}
\right\rangle,
$$
and
$$
D_2^{(2)} =
\left\langle
\begin{pmatrix}
-1 & 0\\0 & -1
\end{pmatrix},
\begin{pmatrix}
1 & -1\\0 & -1
\end{pmatrix}
\right\rangle.
$$
Thus we consider three separate cases, i.e., \mbox{$N = \bm D_2$}, \mbox{$N = D_2^{(1)}$}, and \mbox{$N = D_2^{(2)}$}.
\begin{enumerate}
\item Suppose that $D_2 = A\bm D_2A^{-1}$ for some $A \in G$. As it was explained previously, every matrix $A$ which satisfies $D_2 = A\bm D_2A^{-1}$ must be of the form
$$
\begin{pmatrix}
a & -a\\b & b
\end{pmatrix}
\,\,\,
\textrm{or}
\,\,\,
\begin{pmatrix}
a & a\\b & -b
\end{pmatrix}
$$
for some non-zero $a, b \in \Q$. Therefore,
\begin{align*}
H	& = 
\left\langle
A
\begin{pmatrix}
0 & 1\\
1 & 0
\end{pmatrix}
A^{-1},
A
\begin{pmatrix}
0 & 1\\
-1 & 1
\end{pmatrix}
A^{-1}
\right\rangle\\
& =
\left\langle
\begin{pmatrix}
-1 & 0\\
0 & 1
\end{pmatrix},
\begin{pmatrix}
1/2 & a/(2b)\\
-3b/(2a) & 1/2
\end{pmatrix}
\right\rangle
\end{align*}
Upon setting $t = a/b$, the result follows.

\item Suppose that $D_2 = AD_2^{(1)}A^{-1}$ for some $A \in G$. A straightforward calculation shows that $A$ must be of the form
$$
\begin{pmatrix}
a & -2a\\b & 0
\end{pmatrix}
\,\,\,
\textrm{or}
\,\,\,
\begin{pmatrix}
a & 0\\b & -2b
\end{pmatrix}
$$
Therefore,
$$
H =
\left\langle
\begin{pmatrix}
-1/2 & -3a/(2b)\\
-b/(2a) & 1/2
\end{pmatrix},
\begin{pmatrix}
1/2 & 3a/(2b)\\
-b/(2a) & 1/2
\end{pmatrix}
\right\rangle
$$
Upon setting $t = 3a/b$, the result follows.

\item Suppose that $D_2 = A\bm D_2^{(2)}A^{-1}$ for some $A \in G$. A straightforward calculation shows that $A$ must be of the form
$$
\begin{pmatrix}
-2a & a\\0 & b
\end{pmatrix}
\,\,\,
\textrm{or}
\,\,\,
\begin{pmatrix}
0 & a\\-2b & b
\end{pmatrix}
$$
for some non-zero $a, b \in \Q$. Therefore,
$$
H =
\left\langle
\begin{pmatrix}
-1/2 & -3a/(2b)\\
-b/(2a) & 1/2
\end{pmatrix},
\begin{pmatrix}
1/2 & -3a/(2b)\\
b/(2a) & 1/2
\end{pmatrix}
\right\rangle
$$
Upon setting $t = 3a/b$, the result follows.
\end{enumerate}
\end{enumerate}
\end{proof}

\begin{lem} \label{lem:stuv}
Let $n$ be an odd positive integer such that $\varphi(n) \geq 10$. Let $j$ be an integer coprime to $n$. If
\begin{equation} \label{eq:root_relation}
2\cos\left(\frac{2\pi j}{n}\right) = \frac{2\cos\left(\frac{2\pi}{n}\right)v - u}{-2\cos\left(\frac{2\pi}{n}\right)t + s}
\end{equation}
for some rationals $s$, $t$, $u$ and $v$, then $s \neq 0$, $s = v$ and $t = u = 0$.
\end{lem}

\begin{proof}
For an integer $i$, let $\alpha_i = 2\cos\left(\frac{2\pi i}{n}\right)$. Put $\alpha = \alpha_1$. Since $n$ is odd, it follows from Lemma \ref{lem:galois} that there exists a field automorphism $\sigma_2$ in the Galois group of $\mathbb Q(\alpha)$ such that $\sigma_2(\alpha_\ell) = \alpha_{2\ell}$ for each $\ell$ coprime to $n$. Therefore,
$$
\alpha_{2j} = \sigma_2(\alpha_j) = \sigma_2\left(\frac{v\alpha - u}{-t\alpha + s}\right) = \frac{v\sigma_2(\alpha) - u}{-t\sigma_2(\alpha) + s} = \frac{v\alpha_2 - u}{-t\alpha_2 + s}.
$$
Since for any $x \in \R$  it is the case that $2\cos(2x) = (2\cos(x))^2 - 2$, we conclude that $\alpha_{2i} = \alpha_i^2 - 2$ for all $i$. Therefore,
$$
\left(\frac{v\alpha - u}{-t\alpha + s}\right)^2 - 2 = \alpha_j^2 - 2 = \alpha_{2j} = \frac{v\alpha_2 - u}{-t\alpha_2 + s} = \frac{v(\alpha^2 - 2) - u}{-t(\alpha^2 - 2) + s}.
$$
From the above equality we obtain
$$
\left(-t(\alpha^2 - 2) + s\right)\left((v\alpha - u)^2 - 2(-t\alpha + s)^2\right)
= (-t\alpha + s)^2\left(v(\alpha^2 - 2) - u\right).
$$
We conclude that the polynomial
\begin{align*}
(2 t^3 - t^2 v  - t v^2)x^4\\
+(- 4 s t^2  + 2 s t v  + 2 t u v)x^3\\
+ (2 s^2 t  - s^2 v  - 2 s t^2  + s v^2  - 4 t^3 + t^2 u + 2 t^2 v  - t u^2  + 2 t v^2)x^2\\
+ (4 s^2 t  + 8 s t^2 - 2 s t u  - 4 s t v - 2 s u v  - 4 t u v)x\\
+ (-2 s^3  - 4 s^2 t  + s^2 u  + 2 s^2 v  + s u^2 + 2 t u^2)
\end{align*}
vanishes at $\alpha$. Since the degree of $\alpha$ is $\frac{\varphi(n)}{2} \geq 5$ and the above polynomial has degree at most $4$, it must be the case that this polynomial is identically equal to zero. That is,
\begin{equation} \label{eq:system_of_equations}
\begin{array}{r l}
t(t - v)(2t + v) & = 0,\\
t(- 2 s t  + s v  + u v) & = 0,\\
2 s^2 t  - s^2 v  - 2 s t^2  + s v^2  - 4 t^3 + t^2 u + 2 t^2 v  - t u^2  + 2 t v^2 & = 0,\\
2 s^2 t  + 4 s t^2 - s t u  - 2 s t v - s u v  - 2 t u v & = 0,\\
-2 s^3  - 4 s^2 t  + s^2 u  + 2 s^2 v  + s u^2 + 2 t u^2 & = 0.
\end{array}
\end{equation}
Depending on the value of $t$, we consider the following three cases.

\begin{enumerate}[1.]
\item Suppose that $t = 0$. Then the first two equations in (\ref{eq:system_of_equations}) vanish, while the third and the fourth equations simplify to $sv(v - s) = 0$ and $suv = 0$, respectively. Note that $s \neq 0$, for otherwise the denominator of (\ref{eq:root_relation}) vanishes. Thus the last two equations further reduce to $v(v - s) = 0$ and $uv = 0$. If $v = 0$, then the number $\alpha_j = -u/s$ is rational, in contradiction to the fact that $\deg \alpha_j \geq 5$. Thus it must be the case that $v \neq 0$, $s = v$ and $u = 0$.

\item Suppose that $t = v$ and $t \neq 0$. Then the second equation in (\ref{eq:system_of_equations}) simplifies to $v(u - s) = 0$. But then $s = u$, and
$$
\alpha_j = \frac{v\alpha - u}{-t\alpha + s} = \frac{v\alpha - u}{-v\alpha + u} = -1,
$$
in contradiction to the fact that $\deg \alpha_j \geq 5$.

\item Suppose that $v = -2t$ and $t \neq 0$. Then the second equation simplifies to $v(2s + u) = 0$. But then $u = -2s$,
$$
\alpha_j = \frac{v\alpha - u}{-t\alpha + s} = \frac{-2t\alpha + 2s}{-t\alpha + s} = 2,
$$
in contradiction to the fact that $\deg \alpha_j \geq 5$.
\end{enumerate}
\end{proof}

\begin{lem} \label{lem:reciprocal}
For a positive integer $n$, the minimal polynomial $\Psi_n(x)$ of $2\cos\left(\frac{2\pi}{n}\right)$ is reciprocal \mbox{if and only if} $n = 3$ or $n = 24$.
\end{lem}

\begin{proof}
With the Mathematica command \texttt{MinimalPolynomial[2*cos(2*Pi/n)]} we can compute $\Psi_n$ for every $1 \leq n \leq 24$ and verify that reciprocal polynomials appear only for $n = 3$ and $n = 24$. They are \mbox{$x + 1$} and \mbox{$x^4 - 4x^2 + 1$}.

Next, we implement the formulas
\small
$$
\operatorname{tr}(n) = \sum\limits_{\substack{1 \leq k  < \frac{n}{2}\\ \gcd(k,n) = 1}}2\cos\left(\frac{2\pi k}{n}\right), \quad \operatorname{norm}(n) = \prod\limits_{\substack{1 \leq k  < \frac{n}{2}\\ \gcd(k,n) = 1}}2\cos\left(\frac{2\pi k}{n}\right),
$$
$$
\operatorname{rtr}(n) = \operatorname{norm}(n) \sum\limits_{\substack{1 \leq k  < \frac{n}{2}\\ \gcd(k,n) = 1}}\left(2\cos\left(\frac{2\pi k}{n}\right)\right)^{-1}
$$
\normalsize
in the computer algebra system PARI/GP. Notice that if a polynomial $\Psi_n$ is reciprocal, then $|\operatorname{tr}(n)| = |\operatorname{rtr}(n)|$. Running our PARI/GP code, we can verify that for $25 \leq n \leq 745$ the equality $|\operatorname{tr}(n)| = |\operatorname{rtr}(n)|$ occurs only when $4 \mid n$.

From Part 1 of Lemma \ref{lem:odd_and_even} we know that if $4 \mid n$, then $\Psi_n(x) = g(x^2)$, where $g(x)$ is the minimal polynomial of $2 + 2\cos\left(\frac{4\pi}{n}\right)$. Thus $\Psi_n(x)$ is reciprocal if and only if $g(x)$ is reciprocal. We then implement the formulas
\small
$$
\operatorname{tr}'(n) = \sum\limits_{\substack{1 \leq k  < \frac{n}{4}\\ \gcd(k,n/2) = 1}}\left(2 + 2\cos\left(\frac{2\pi k}{n}\right)\right), \quad \operatorname{norm}'(n) = \prod\limits_{\substack{1 \leq k  < \frac{n}{4}\\ \gcd(k,n/2) = 1}}\left(2 + 2\cos\left(\frac{2\pi k}{n}\right)\right),
$$
$$
\operatorname{rtr}'(n) = \operatorname{norm}'(n) \sum\limits_{\substack{1 \leq k  < \frac{n}{4}\\ \gcd(k,n/2) = 1}}\left(2 + 2\cos\left(\frac{2\pi k}{n}\right)\right)^{-1}
$$
\normalsize
in PARI/GP. Notice that if a polynomial $\Psi_n$ with $4 \mid n$ is reciprocal, then $|\operatorname{tr}'(n)| = |\operatorname{rtr}'(n)|$. Running our PARI/GP code, we can verify that $|\operatorname{tr}'(n)| \neq |\operatorname{rtr}'(n)|$ for all $4 \mid n$ such that $25 \leq n \leq 745$. We conclude that every polynomial $\Psi_n(x)$ with $25 \leq n \leq 745$ is not reciprocal.

It remains to prove that there are no reciprocal polynomials with $n \geq 746$. For a positive integer $n$, let $g(n)$ denote the Jacobsthal's function; that is, $g(n)$ is equal to the smallest positive integer $m$ such that every sequence of $m$ consecutive integers contains an integer coprime to $n$. It was proven by Kanold \cite{kanold} that
$$
g(n) \leq 2^{\omega(n)},
$$
where $\omega(n)$ denotes the number of distinct prime factors of $n$.\footnote{The author is grateful to Prof.\ Jeffrey Shallit for pointing out that better bounds exist, e.g., \cite{iwaniec, vaughan}. However, Kanold's bound is sufficient for our purposes.} Combining the above upper bound with the inequality \cite{robin}
$$
\omega(n) \leq 1.3841\frac{\log n}{\log\log n},
$$
which holds for all $n \geq 3$, we get
$$
g(n) < n^{\frac{0.96}{\log \log n}}.
$$

Now, consider the interval $\left[\frac{1}{2\pi}\arccos\left(\frac{1}{4}\right), \frac{1}{4}\right)$. We claim that this interval contains a rational number $j/n$ with $j$ coprime to $n$. In other words, we would like to locate an integer $j$ coprime to $n$ such that
$$
\frac{1}{2\pi}\arccos\left(\frac{1}{4}\right)n \leq j < \frac{1}{4}n.
$$
We see that such an integer $j$ has to belong to the interval $\left[\frac{1}{2\pi}\arccos\left(\frac{1}{4}\right)n, \frac{1}{4}n\right)$, whose length exceeds $n/25$. Since our interval is half-closed, it contains at least $\lfloor n/25 \rfloor$ consecutive integers. However, for all $n \geq 746$ we have
$$
n^{\frac{0.96}{\log \log n}} < \frac{n}{25} - 1,
$$
and this inequality implies that
$$
g(n) < n^{\frac{0.96}{\log \log n}} < \frac{n}{25} - 1 < \left\lfloor\frac{n}{25}\right\rfloor.
$$
This means that the interval $\left[\frac{1}{2\pi}\arccos\left(\frac{1}{4}\right)n, \frac{1}{4}n\right)$ contains an integer $j$ that is coprime to $n$. But then
$$
\arccos\left(\frac{1}{4}\right) \leq \frac{2\pi j}{n} < \frac{\pi}{2},
$$
and consequently
$$
0 < \alpha_j \leq \frac{1}{2}.
$$
If we now assume that $\Psi_n(x)$ is reciprocal, then the number $\alpha_j^{-1}$ is a conjugate of $\alpha_j$, so there exists some $\ell$ such that $\alpha_\ell = \alpha_j^{-1}$. Thus $\alpha_\ell \geq 2$. On the other hand, $\alpha_\ell \leq 2$, which means that $\ell = 0$. Since $\gcd(\ell, n) = 1$, we conclude that $n = 1$, and this contradicts our assumption that $n \geq 746$.
\end{proof}

\begin{lem} \label{lem:when-fields-are-the-same}
Let $k$ and $\ell$ be positive integers such that $k < \ell$ and $k, \ell \notin \{1, 2,3, 4, 6\}$. Then $\mathbb Q\left(2\cos\left(\frac{2\pi}{k}\right)\right) = \mathbb Q\left(2\cos\left(\frac{2\pi}{\ell}\right)\right)$ if and only if $k$ is odd and $\ell = 2k$.
\end{lem}

\begin{proof}
Suppose that $k$ is odd and $\ell = 2k$. Then the degrees of $\mathbb Q\left(2\cos\left(\frac{2\pi}{k}\right)\right)$ and $\mathbb Q\left(2\cos\left(\frac{2\pi}{\ell}\right)\right)$ are both equal to $\frac{\varphi(k)}{2}$. Further,
$$
2\cos\left(\frac{2\pi}{k}\right) = \left(2\cos\left(\frac{\pi}{k}\right)\right)^2 - 2 = \left(2\cos\left(\frac{2\pi}{\ell}\right)\right)^2 - 2,
$$
which means that $\mathbb Q\left(2\cos\left(\frac{2\pi}{k}\right)\right) \subseteq \mathbb Q\left(2\cos\left(\frac{2\pi}{\ell}\right)\right)$. Since $\mathbb Q\left(2\cos\left(\frac{2\pi}{k}\right)\right)$ is a subfield of $\mathbb Q\left(2\cos\left(\frac{2\pi}{\ell}\right)\right)$ of the same degree, it must be the case that the two fields are identical.

Conversely, suppose that $\mathbb Q\left(2\cos\left(\frac{2\pi}{k}\right)\right) = \mathbb Q\left(2\cos\left(\frac{2\pi}{\ell}\right)\right)$. Then degrees of these number fields are equal, i.e., $\frac{\varphi(k)}{2} = \frac{\varphi(\ell)}{2}$. We claim that there exists a prime that divides $k$ but not $\ell$, or vice versa. For suppose that this is not the case and
$$
k = \prod_{i = 1}^{t}p_i^{e_i}, \,\,\, \ell = \prod_{i = 1}^{t}p_i^{f_i}
$$
for some positive integers $t, e_1, \ldots, e_t, f_1, \ldots, f_t$ and distinct primes $p_1, \ldots, p_t$. Then
$$
\prod_{i = 1}^tp_i^{e_i - 1}(p_i - 1) = \varphi(k) = \varphi(\ell) = \prod_{i = 1}^tp_i^{f_i - 1}(p_i - 1).
$$
After dividing both sides by $\prod_{i = 1}^t(p_i - 1)$, we obtain
$$
\prod_{i = 1}^tp_i^{e_i - 1} = \prod_{i = 1}^tp_i^{f_i - 1},
$$
which means that $e_i = f_i$ for all $i = 1, 2, \ldots, t$. But then $k = \ell$, in contradiction to our assumption that $k$ and $\ell$ are distinct. This completes the proof of our claim.

Now, it follows from the result of Lehmer \cite[Theorem 3.8]{lehmer} that the discriminant $D_k$ of the field $\mathbb Q\left(2\cos\left(\frac{2\pi}{k}\right)\right)$ can be computed as follows:
\begin{equation} \label{eq:disc}
D_k =
\begin{cases}
2^{(j-1)2^{j - 2} - 1} & \textrm{if $k = 2^j$, $j > 2$,}\\
p^{(jp^j - (j+1)p^{j-1} - 1)/2} & \textrm{if $k = p^j$ or $2p^j$, $p > 2$ prime,}\\
\left(\prod_{i = 1}^{\omega(k)}p_i^{e_i - 1/(p_i - 1)}\right)^{\frac{\varphi(k)}{2}} & \textrm{if $\omega(k) > 1, k \neq 2p^j$.}
\end{cases}
\end{equation}
An analogous formula applies to the discriminant $D_\ell$ of $\mathbb Q\left(2\cos\left(\frac{2\pi}{\ell}\right)\right)$, and of course we must have $D_\ell = D_k$. Suppose that $k = 2^j$, $j > 2$. Then $D_k$ is a power of $2$. Since there is a prime that divides $k$ but not $\ell$ or vice versa, it must be the case that an odd prime $q$ divides $\ell$. But then it follows from (\ref{eq:disc}) that $q \mid D_\ell$, so $D_k \neq D_\ell$. Thus, this case is impossible, and so at least one odd prime divides $k$, i.e.,
$$
k = 2^r\prod\limits_{i = 1}^tp_i^{e_i}
$$
for some non-negative integer $r$, positive integers $t, e_1, \ldots, e_t$, and distinct odd primes $p_1, \ldots, p_t$.

Notice how in (\ref{eq:disc}), for every odd prime $p$, $p \mid k$ if and only if $p \mid D_k$. Similarly, for every odd prime $q$, $q \mid \ell$ if and only if $q \mid D_\ell$. Since $D_k = D_\ell$, we conclude that, for every odd prime $p$, $p \mid k$ if and only if $p \mid \ell$. Thus,
$$
\ell = 2^s\prod\limits_{i = 1}^{t}p_i^{f_i}
$$
for some non-negative integer $s$ and positive integers $f_1, \ldots, f_t$. Further, since there exists a prime that divides $k$ but not $\ell$ or vice versa, it must be the case that either $k$ or $\ell$ is odd. At this point, we consider four cases.

\begin{enumerate}
\item If $k = p^j$ for some odd prime $p$, then it follows from (\ref{eq:disc}) that $D_k$ is odd. Further, $\ell = 2^sp^m$ for some positive integers $s$ and $m$ (recall that there must be a prime that divides $\ell$, but not $k$). Further, it must be the case that $s = 1$, for otherwise it follows from (\ref{eq:disc}) that $D_\ell$ is even. Thus, $\ell = 2p^m$, and so
$$
p^{(jp^j - (j+1)p^{j-1} - 1)/2} = D_k = D_\ell = p^{(mp^m - (m+1)p^{m-1} - 1)/2}.
$$
Since the function $f_p(x) = xp^x - (x + 1)p^{x - 1} - 1$ is monotonously increasing on the interval $[1, +\infty)$, we conclude that $m = j$, and so $\ell = 2p^j = 2k$.

\item If $k = 2p^j$ for some odd prime $p$, then it follows from (\ref{eq:disc}) that $D_k$ is odd. Further, $\ell = 2^sp^m$ for some non-negative integer $s$ and positive integer $m$. Since there must exist a prime that divides $k$ but not $\ell$, we conclude that $\ell = 2^s \geq 8$ or $\ell = p^m$. The former is impossible, since $D_\ell$ has to be odd. Thus, $\ell = p^m$, and so
$$
p^{(jp^j - (j+1)p^{j-1} - 1)/2} = D_k = D_\ell = p^{(mp^m - (m+1)p^{m-1} - 1)/2}.
$$
Since the function $f_p(x) = xp^x - (x + 1)p^{x - 1} - 1$ is monotonously increasing on the interval $[1, +\infty)$, we conclude that $m = j$. But then $\ell = p^j < 2p^j = k$, which is impossible, since we assumed that $k < \ell$.

\item If $k$ is odd and it is not an odd prime power, then $t \geq 2$ and
$$
D_k = \left(\prod_{i = 1}^{t}p_i^{e_i - 1/(p_i - 1)}\right)^{\frac{\varphi(k)}{2}}.
$$
Since $D_k$ is odd, $D_\ell$ is odd, which in turn implies that $\ell = 2^s\prod_{i = 1}^{t}p_i^{f_i}$ for $s \in \{0, 1\}$. Thus,
$$
D_\ell = \left(\prod_{i = 1}^{t}p_i^{f_i - 1/(p_i - 1)}\right)^{\frac{\varphi(\ell)}{2}}.
$$
Since $D_k = D_\ell$, the unique factorization tells us that $\frac{\varphi(k)}{2}\left(e_i - \frac{1}{p_i - 1}\right) = \frac{\varphi(\ell)}{2}\left(f_i - \frac{1}{p_i - 1}\right)$ for every $i = 1, \ldots, t$. Since $\frac{\varphi(k)}{2} = \frac{\varphi(\ell)}{2}$, we conclude that $e_i = f_i$ for every $i = 1, \ldots, t$. Since $k \neq \ell$, we conclude that $s = 1$, and so $\ell = 2k$.

\item If $k = 2^r\prod_{i = 1}^{t}p_i^{e_i}$ is even and it is not twice an odd prime power, then $\ell$ must be odd. Consequently, $D_\ell$ is odd. But then $D_k = D_\ell$ is odd, which is only possible when $r = 1$. Thus,
$$
\left(\prod_{i = 1}^{t}p_i^{e_i - 1/(p_i - 1)}\right)^{\frac{\varphi(k)}{2}} = D_k = D_\ell = \left(\prod_{i = 1}^{t}p_i^{f_i - 1/(p_i - 1)}\right)^{\frac{\varphi(\ell)}{2}}.
$$
Once again, we find that $e_i = f_i$ for all $i = 1, \ldots, t$, meaning that $k = 2\prod_{i = 1}^{t}p_i^{e_i}$ and $\ell = \prod_{i = 1}^{t}p_i^{e_i} = \frac{k}{2} < k$, which contradicts $k < \ell$.
\end{enumerate}

\end{proof}

\section{Automorphisms of $\Psi_n(x, y)$} \label{sec:cos}

In this section we prove Theorem \ref{thm:cos}. Let $n$ be a positive integer such that $n \notin \{1, 2, 3, 4, 5, 6, 8, 10, 12\}$ and let $d = \varphi(n)/2$, so that $\deg \Psi_n = d$ and $d \geq 3$. In Sections \ref{subsec:1}, \ref{subsec:2} and \ref{subsec:3} we consider three cases separately:
\begin{itemize}
\item $d \geq 4$ and $n \equiv 0 \pmod 4$;

\item $d \geq 5$ and $n \not \equiv 0 \pmod 4$; and

\item $d = 3, 4$ and $n \not \equiv 0 \pmod 4$.
\end{itemize}

\subsection{Case $d \geq 4$ and $n \equiv 0 \pmod 4$} \label{subsec:1}

Let $n \geq 16$ be an integer such that $n \equiv 0 \pmod 4$. Then it follows from \mbox{Part 1} of Lemma \ref{lem:odd_and_even} that $\Psi_n(x) = g(x^2)$ for some $g(x) \in \mathbb Z[x]$. Consequently, there exists a binary form $G(x, y) \in \mathbb Z[x, y]$ such that $\Psi_n(x, y) = G(x^2, y^2)$. Therefore,
$$
D_2 = 
\left\langle
\begin{pmatrix}-1 & 0\\0 & 1\end{pmatrix},
\begin{pmatrix}1 & 0\\0 & -1\end{pmatrix}
\right\rangle
$$
is a subgroup of $\Aut \Psi_n$.

We claim that $D_2$ is a \emph{proper} subgroup of $\Aut |\Psi_n|$ if and only if \mbox{$n = 24$}. Since $D_2 \subseteq \Aut \Psi_n \subseteq \Aut |\Psi_n|$, this result would imply that $\Aut \Psi_n = \Aut |\Psi_n| = D_2$ for any positive integer $n \geq 16$ such that $n \equiv 0$ \mbox{(mod $4$)} and $n \neq 24$.

By Lemma \ref{lem:options}, if $D_2$ is a proper subgroup of $\Aut |\Psi_n|$, then there exists a non-zero $t \in \mathbb Q$ such that
$$
\Aut |\Psi_n| \cong \left\langle
\begin{pmatrix}
-1 & 0\\
0 & 1
\end{pmatrix},
\begin{pmatrix}
0 & t\\
-1/t & 0
\end{pmatrix}
\right\rangle
$$
or
$$
\Aut |\Psi_n| \cong \left\langle
\begin{pmatrix}
-1 & 0\\
0 & 1
\end{pmatrix},
\begin{pmatrix}
1/2 & t/2\\
-3/(2t) & 1/2
\end{pmatrix}
\right\rangle.
$$
We will consider these two options separately. In each case, we will make use of the formula
\begin{equation} \label{eq:constant_coefficient}
|\Psi_m(0)| =
\begin{cases}
0 & \textrm{if $m = 4$,}\\
2 & \textrm{if $m = 2^k$ for $k \geq 3$,}\\
p & \textrm{if $m = 4p^k$ for $k \geq 1$, where $p$ is an odd prime,}\\
1 & \textrm{otherwise.}
\end{cases}
\end{equation}
The proof of (\ref{eq:constant_coefficient}) can be found in \cite{demirci-cangul}.

\begin{enumerate}[1.]
\item Suppose that there exist integers $a \neq 0$ and $b \geq 1$ such that $\gcd(a, b) = 1$ and $M \in \Aut |\Psi_n|$, where
$$
M = \left(
\begin{matrix}
0 & a/b\\
-b/a & 0
\end{matrix}
\right).
$$
Then
\begin{align*}
\Psi_n(x, y)
& = \pm \Psi_n\left(\frac{a}{b}y, -\frac{b}{a}x\right)\\
& = \pm (ab)^{-d}\Psi_n\left(a^2y, -b^2x\right).
\end{align*}
Thus
$$
(ab)^d\Psi_n(x,y) = \pm\Psi_n(a^2y, -b^2x).
$$
By plugging $x = 1$ and $y = 0$ into the above equation, we see that \mbox{$c_0 = \pm (a/b)^d$}, where $c_0$ denotes the constant coefficient of $\Psi_n(x)$. Since $c_0$ is an integer, it must be the case that $t = a/b$ is an integer such that $t^d = \pm c_0$. By (\ref{eq:constant_coefficient}), the value of $|c_0|$ is squarefree, and since $d \geq 2$ is even and $t^d = |c_0|$, we conclude that $c_0 = \pm 1$. Therefore, $t = a/b = \pm 1$, which means that $M = \left(\begin{smallmatrix}0 & 1\\- 1 & 0\end{smallmatrix}\right)$ or $M = \left(\begin{smallmatrix}0 & -1\\1 & 0\end{smallmatrix}\right)$. But then
$$
\Aut |\Psi_n| = \left\langle
\begin{pmatrix}
0 & 1\\
1 & 0
\end{pmatrix},
\begin{pmatrix}
0 & 1\\
-1 & 0
\end{pmatrix}
\right\rangle.
$$

Now, suppose that there exists $A$ in $\Aut |\Psi_n|$ such that $A \notin \Aut \Psi_n$. Plugging $x = 0$ and $y = 1$ into $\Psi_n(x, y) = -\Psi_n(y, -x)$, we find that $c_0 = -1$. It follows from (\ref{eq:constant_coefficient}) that $n \neq 4p^k$ for any odd prime $p$ and any positive integer $k$. From Part 1 of Lemma \ref{lem:odd_and_even} we know that $\Psi_n(\alpha) = 0$ if and only if $\Psi_n(-\alpha) = 0$, so
\small
$$
c_0 = \prod\limits_{\substack{1 \leq k < \frac{n}{2}\\ \gcd(k, n) = 1}}\left(2\cos\left(\frac{2\pi k}{n}\right)\right) = \prod\limits_{\substack{1 \leq k < \frac{n}{4}\\ \gcd(k, n) = 1}}\left(-4\cos\left(\frac{2\pi k}{n}\right)^2\right).
$$
\normalsize
Since $c_0$ is negative, the number $N$ of integers in the interval $[1, n/4)$ that are coprime to $n$ must be odd. If we write $n = 2^s t$, where $s \geq 2$ is an integer and $t$ is odd, then through elementary number theoretic observations we find that $N = \frac{d}{2} = 2^{s - 2}\frac{\varphi(t)}{2}$. Since $N$ is odd, we find that $s = 2$ and $t = p^k$ for some prime $p \equiv 3$ \mbox{(mod $4$)}, in contradiction to the fact that $n \neq 4p^k$ for any odd prime $p$ and any positive integer $k$. Hence $\Aut |\Psi_n| = \Aut \Psi_n$, and so $\Psi_n(x, y) = \Psi_n(y, x)$, which makes $\Psi_n(x)$ a reciprocal polynomial. It follows from Lemma \ref{lem:reciprocal} that $n = 24$ and $\Psi_n(x, y) = x^4 - 4x^2y^2 + y^4$.

\item Suppose that there exist integers $a \neq 0$ and $b \geq 1$ such that $\gcd(a, b) = 1$ and $M \in \Aut |\Psi_n|$, where
$$
M =
\begin{pmatrix}
1/2 & a/(2b)\\
-3b/(2a) & 1/2
\end{pmatrix}.
$$
We will show that this is impossible.

Since $M \in \Aut |\Psi_n|$,
\begin{align*}
\Psi_n(x, y)
& = \pm \Psi_n\left(\frac{1}{2}x + \frac{a}{2b}y, -\frac{3b}{2a}x + \frac{1}{2}y\right)\\
& = \pm (2ab)^{-d}\Psi_n\left(abx + a^2y, -3b^2x + aby\right).
\end{align*}
Thus
\begin{equation} \label{eq:D6_relation}
(2ab)^d\Psi_n(x, y) = \pm \Psi_n\left(abx + a^2y, -3b^2x + aby\right).
\end{equation}
By plugging $x = 0$ and $y = 1$ into the above equation, we obtain \mbox{$c_02^db^d = \pm \Psi_n(a, b)$}. Thus $\Psi_n(a, b)$ is divisible by $b$. Since the leading coefficient of $\Psi_n(x, y)$ is equal to one, we see that
$$
a^d \equiv \Psi(a, 0) \equiv \Psi_n(a, b) \equiv 0 \pmod b.
$$
Then $b \mid a^d$, and since $a$ and $b \geq 1$ are coprime, we conclude that $b = 1$ and $c_02^d = \pm \Psi_n(a)$. By plugging $x = 1$ and $y = 0$ into (\ref{eq:D6_relation}), we obtain $(2a)^d = \pm \Psi_n(a, -3)$. Since $\Psi_n(x) = g(x^2)$, we see that
$$
\pm c_0(-3)^d \equiv \Psi_n(a, -3) \equiv 0 \pmod{a^2},
$$
which means that $a^2 \mid c_03^d$. By (\ref{eq:constant_coefficient}) the value of $c_0$ is squarefree, so $a = \pm 3^r$ for some non-negative integer $r$. Since $\Aut |\Psi_n|$ is a group, we may replace $M$ with
$$
M^{-1} = \begin{pmatrix}
1/2 & -a/(2b)\\
3b/(2a) & 1/2
\end{pmatrix},
$$
and so without loss of generality we may assume that $a = 3^r$.

After plugging $a = 3^r$ and $b = 1$ into (\ref{eq:D6_relation}) we obtain
\begin{equation} \label{eq:D6_relation_with_r}
2^d3^{(r-1)d}\Psi_n(x, y) = \pm \Psi_n\left(3^{r-1}x + 3^{2r-1}y, -x + 3^{r-1} y\right).
\end{equation}
Suppose that $r \geq 3$. Then
$$
\pm \Psi_n\left(3^{r-1}x + 3^{2r-1}y, -x + 3^{r-1} y\right) \equiv \Psi_n(0, -x) \equiv c_0(-x)^d \equiv 0 \pmod 9.
$$
Since this congruence must hold for all $x$, it holds for those $x$ that are not divisible by $3$, which means that $9$ divides $c_0$. However, this result contradicts (\ref{eq:constant_coefficient}), which states that the value of $c_0$ is squarefree. We conclude that the only possible values of $r$ are $0, 1, 2$, and so the only possible values of $a = 3^r$ are $1, 3$ and $9$. We consider these three cases separately:
\begin{itemize}
\item  For $r = 0$, $x = 0$, $y = 1$ the equation (\ref{eq:D6_relation_with_r}) gives us $|\Psi_n(1)| = 2^d$;

\item For $r = 1$, $x = 1$, $y = 1$ the equation (\ref{eq:D6_relation_with_r}) gives us $|\Psi_n(1)| = 2^d$;

\item For $r = 2$, $x = 1$, $y = 1$  the equation (\ref{eq:D6_relation_with_r}) gives us $|\Psi_n(1)| = 3^{-d}|\Psi_n(15)|$. Since $15 - 2\cos(x) \geq 13$ for any $x \in \R$,
$$
|\Psi_n(1)| = 3^{-d}|\Psi_n(15)|
 = 3^{-d}\prod\limits_{\substack{1 \leq j < n/2\\\gcd(j, n) = 1}}\left|15 - 2\cos\left(\frac{2\pi j}{n}\right)\right|
 \geq 3^{-d}13^d > 2^d.
$$
\end{itemize}
Thus, regardless of the value of $r$, we must have $|\Psi_n(1)| \geq 2^d$. We claim that the opposite is always true, i.e., $|\Psi_n(1)| < 2^{d}$.

To see that this is the case, first assume that $n \geq 14336$. Let $\Phi_n(x)$ denote the $n$-th cyclotomic polynomial, and recall Lehmer's identity
$$
\Psi_n(z + z^{-1}) = z^{-d}\Phi_n(z),
$$
which holds for every non-zero $z \in \mathbb C$. Then
$$
\Psi_n(1) = \Psi_n\left(2\cos\left(\frac{\pi}{3}\right)\right) = \Psi_n\left(e^{\frac{\pi i}{3}} + e^{-\frac{\pi i}{3}}\right) = e^{-\frac{d\pi i}{3}}\Phi_n\left(e^{\frac{\pi i}{3}}\right).
$$
As a consequence of this,
$$
|\Psi_n(1)| = \left|\Phi_n\left(e^{\frac{\pi i}{3}}\right)\right| \leq L(\Phi_n)\max\left\{1, \left|e^{\frac{\pi i}{3}}\right|\right\} = L(\Phi_n),
$$
where $L(\Phi_n)$ is the sum of absolute values of coefficients of $\Phi_n$. By \mbox{\cite[Lemme 4.1]{fouvry-waldschmidt}}, it is the case that $L(\Phi_n) \leq n^{\frac{\sigma_0(n)}{2}}$, where $\sigma_0(n)$ is the number of positive divisors of $n$. We conclude that $|\Psi_n(1)| \leq n^{\frac{\sigma_0(n)}{2}}$. By \cite{nicolas-robin},
$$
\sigma_0(n) \leq n^{\frac{1.067}{\log \log n}}.
$$
By \cite[Theorem 15]{rosser-schoenfeld},
$$
\varphi(n) > \frac{n}{5 \log \log n}.
$$
Since $n^{\frac{1.067}{\log\log n}}\log n < \frac{\log 2}{2} \frac{n}{5\log \log n}$ for all $n \geq 14336$, we find that
\begin{align*}
|\Psi_n(1)|
& \leq L(\Phi_n)\\
& \leq \exp\left(\frac{1}{2}\sigma_0(n)\log n\right)\\
& \leq \exp\left(\frac{1}{2}n^{\frac{1.067}{\log\log n}}\log n\right)\\
& < \exp\left(\frac{\log 2}{2} \cdot \frac{n}{5\log\log n}\right)\\
& < \exp\left(\frac{\log 2}{2}\varphi(n)\right)\\
& = 2^{d}.
\end{align*}

It remains to check that $|\Psi_n(1)| < 2^d$ for $16 \leq n \leq 14335$. Since $|\Psi_n(1)| = \left|\Phi_n\left(e^{\frac{\pi i}{3}}\right)\right|$, this fact can be verified with the following PARI/GP code:
\begin{verbatim}
for (n = 16, 14335,
     if ( abs(polcyclo(n, exp(Pi*I/3))) >= 2^(eulerphi(n)/2),
          print(n)
     )
)
\end{verbatim}
%
Since the above code does not print out any integers, we conclude that the relation (\ref{eq:D6_relation}) is impossible, and so neither $\Aut \Psi_n$ nor $\Aut |\Psi_n|$ are isomorphic to $\mathbf D_6$.
\end{enumerate}

\subsection{Case $d \geq 5$ and $n \not \equiv 0 \pmod 4$}  \label{subsec:2}

If we let $S = \left(\begin{smallmatrix}-1 & 0\\0 & 1\end{smallmatrix}\right)$, then Part 2 of Lemma \ref{lem:odd_and_even} tells us that $\Psi_{2n} = (-1)^{\varphi(n)/2}(\Psi_n)_S$ for any odd integer $n \geq 3$. By Lemma \ref{lem:automorphism_of_adjacent_form},
$$
\Aut \Psi_{2n} = S^{-1}(\Aut \Psi_n)S \quad \text{and} \quad \Aut |\Psi_{2n}| = S^{-1}(\Aut |\Psi_n|)S.
$$
In other words, we  can derive $\Aut \Psi_{2n}$ and $\Aut |\Psi_{2n}|$ from $\Aut \Psi_n$ and $\Aut |\Psi_n|$. Therefore, we may assume that $n$ is odd.

Let $M = \frac{1}{m}\left(\begin{smallmatrix}s & u\\t & v\end{smallmatrix}\right)$ be an element of $\Aut |\Psi_n|$, where $s, t, u, v$ and $m = \sqrt{|sv - tu|}$ are integers such that $\gcd(s, t, u, v) = 1$. Then
$$
m^d\Psi_n(x, y) = \Psi_n(sx + uy, tx + vy),
$$
which means that the polynomials $m^d\Psi_n(x)$ and $\Psi_n(sx + u, tx + v)$ are equal. For an integer $\ell$, let $\alpha_\ell = 2\cos\left(\frac{2\pi \ell}{n}\right)$, and let $\alpha = \alpha_1$. Then
\begin{align*}
m^d\Psi_n(sx + u, tx + v)
& = \pm \prod\limits_{\substack{1 \leq \ell < n/2\\\gcd(\ell, n) = 1}}\left((sx + u) - \alpha_\ell(tx + v)\right)\\
& = \pm \prod\limits_{\substack{1 \leq \ell < n/2\\\gcd(\ell, n) = 1}}\left((-t\alpha_\ell + s)x - (v\alpha_\ell - u)\right)\\
& = \pm \Psi_n(s, t)\prod\limits_{\substack{1 \leq \ell < n/2\\\gcd(\ell, n) = 1}}\left(x - \frac{v\alpha_\ell-u}{-t\alpha_\ell + s}\right).
\end{align*}
Since $\Psi_n(x)$ and $\Psi_n(sx + u, tx + v)$ have the same roots, we conclude that there exists some $j$ coprime to $n$ such that
$$
\alpha_j = \frac{v\alpha - u}{-t\alpha + s}.
$$
It follows from Lemma \ref{lem:stuv} that $s \neq 0$, $s = v$ and \mbox{$t = u = 0$}. Since $s$, $t$, $u$ and $v$ are integers such that $\gcd(s, t, u, v) = 1$, we find that $\gcd(s, v) = 1$. This means that $s = v = \pm 1$ and \mbox{$M \in \{\pm I\} \subseteq \Aut |\Psi_n|$}. Therefore, $\Aut |\Psi_n| \subseteq \{\pm I\}$. Thus,
\begin{itemize}
\item if $d$ is odd, then $\Aut \Psi_n = \{I\}$ and $\Aut |\Psi_n| = \{\pm I\}$; and
\item if $d$ is even, then $\Aut \Psi_n = \Aut |\Psi_n| = \{\pm I\}$.
\end{itemize}

\subsection{Case $d = 3, 4$ and $n \not \equiv 0 \pmod 4$}  \label{subsec:3}

It remains to consider the cases $d = 3, 4$ and $n \not \equiv 0 \pmod 4$, which correspond to $n \in \{7, 9, 14, 15, 18, 30\}$.

The binary forms $\Psi_7, \Psi_9, \Psi_{14}, \Psi_{18}$ have degree $3$ and their discriminants are
$$
D_{\Psi_{7}} = D_{\Psi_{14}} = 7^2 \quad \text{and} \quad \quad D_{\Psi_9} = D_{\Psi_{18}} = 9^2.
$$
By Part 2 of \cite[Theorem 3.1]{xiao}, if a binary cubic form $F(x, y) = b_3x^3 + b_2x^2y + b_1xy^2 + b_0y^3$ is irreducible and $D_F$ is a square of an integer, then $\Aut F = \langle \mathcal N_q\rangle$ is isomorphic to $\bf C_3$. The matrix $\mathcal N_q$ which generates $\Aut F$ can be determined with the formula
$$
\mathcal N_q = \frac{1}{2D_q}\begin{pmatrix}b\sqrt{-3D_q} - D_q & 2c\sqrt{-3D_q}\\ -2a\sqrt{-3D_q} & -b\sqrt{-3D_q} - D_q\end{pmatrix},
$$
where $q(x, y) = ax^2 + bxy + cy^2$ is the \emph{Hessian} of $F$ of discriminant $D_q$, with coefficients
$$
a = b_2^2 - 3b_3b_1, \quad b = b_2b_1 - 9b_3b_0, \quad c = b_1^2 - 3b_2b_0.
$$
In the case when $F = \Psi_7$, we have $b_3 = 1$, $b_2 = 1$, $b_1 = -2$ and $b_0 = -1$. Thus $q(x, y) = 7x^2 + 7xy + 7y^2$, $D_q = -147$ and $\mathcal N_q = \frac{1}{-294}\left(\begin{smallmatrix}294 & 294\\-294 & 0\end{smallmatrix}\right) = \left(\begin{smallmatrix}-1 & -1\\1 & 0\end{smallmatrix}\right)$. Since $\Aut \Psi_7$ is a normal subgroup of $\Aut |\Psi_7|$ of index at most $2$, $-I \in \Aut |\Psi_7|$ and $-I \notin \Aut \Psi_7$, we find that $\Aut |\Psi_7| = \langle \mathcal N_q, -I\rangle \cong \bf D_3$. The automorphism groups of $\Psi_9$, $\Psi_{14}$ and $\Psi_{18}$ can be determined analogously.

The binary forms $\Psi_{15}$ and $\Psi_{30}$ both have degree $4$. Using the formula provided in \cite[Section 4]{xiao}, we find that $\Psi_{15}$ has degree $6$ covariant
$$
F_6(x, y) = 15(x^2 - 2xy + 2y^2)(x^4 + 6x^3y + 6x^2y^2 - 4xy^3 - 4y^4).
$$
Now, we refer to a binary quadratic form $f$ as \emph{rationally significant} if it is proportional over $\mathbb C$ to a quadratic form $g$ with integer coefficients and $|D_g|$ is a square of an integer. Notice that $f(x, y) = ax^2 + bxy + cy^2$ with $a = 1$, $b = -2$ and $c = 2$ is a unique rationally significant factor of $F_6$. Thus it follows from \cite[Theorem 4.1]{xiao} that $\Aut F$ is generated by $-I$ and
$$
U_f = \frac{1}{\sqrt{|D_f|}}\begin{pmatrix}b & 2c\\-2a & -b\end{pmatrix} = \frac{1}{2}\begin{pmatrix}-2 & 4\\-2 & 2\end{pmatrix} = \begin{pmatrix}-1 & 2\\-1 & 1\end{pmatrix}.
$$
Since $U_f^2 = -I$, we conclude that $\Aut \Psi_{15} = \langle U_f\rangle \cong \bf C_4$.

It remains to determine $\Aut |\Psi_{15}|$. Suppose that there exists $A = \left(\begin{smallmatrix}s & u\\t & v\end{smallmatrix}\right)$ in $\Aut |\Psi_{15}|$ such that $A \notin \Aut \Psi_{15}$. If we let $\alpha_\ell = 2\cos\left(\frac{2\pi \ell}{n}\right)$ and put $\alpha = \alpha_1$, then the roots of $\Psi_{15}(x)$ are $\alpha$, $\alpha_2$, $\alpha_4$ and $\alpha_7$. Since $A \in \Aut |\Psi_{15}|$, there must exist $j \in \{2, 4, 7\}$ such that $\alpha_j = \frac{v\alpha - u}{-t\alpha + s}$.  Since $U_f \in \Aut \Psi_{15}$, we find that $\alpha_4 = \frac{\alpha - 2}{\alpha - 1}$ and $\alpha_7 = \frac{\alpha_2 - 2}{\alpha_2 - 1}$. We also know that $\alpha_2 = \alpha^2 - 2$, so $\alpha_7 = \frac{\alpha^2 - 4}{\alpha^2 - 3}$.  Since $\deg \alpha = 4$, it is straightforward to verify that there are no rational $s, t, u, v$ such that $\frac{v\alpha - u}{-t\alpha + s} = \alpha_2 = \alpha^2 - 2$ or $\frac{v\alpha - u}{-t\alpha + s} = \alpha_7 = \frac{\alpha^2 - 4}{\alpha^2 - 3}$. In turn, the rationals $s, t, u, v$ that satisfy $\frac{v\alpha - u}{-t\alpha + s} = \alpha_4 = \frac{\alpha - 2}{\alpha - 1}$ correspond to $\pm U_f$. Thus it must be the case that $\Aut |\Psi_{15}| = \Aut \Psi_{15} = \langle U_f\rangle$. Since $\Psi_{30}(x, y) = \Psi_{15}(-x, y)$, we can easily determine $\Aut \Psi_{30}$ and $\Aut |\Psi_{30}|$.

\section{Automorphisms of $T_n(x, y)$ and $U_n(x, y)$} \label{sec:chebyshev}

In this section we prove Theorem \ref{thm:chebyshev}. Let $T_n(x, y)$ and $U_n(x, y)$ denote the homogenizations of the $n$-th Chebyshev polynomials of first and second kinds, respectively. Define
$$
\tilde U_n(x, y) = U_{n - 1}\left(\frac{x}{2}, y\right) \quad \text{and} \quad \tilde V_n(x, y) = 2T_n\left(\frac{x}{2}, y\right).
$$
Then for $n \geq 1$ we have
\begin{equation} \label{eq:chebyshev1}
\tilde U_n(x, y) = x^{\frac{1 - (-1)^{n - 1}}{2}}\prod\limits_{\substack{d \mid 2n\\d \notin \{1, 2, 4\}}}\Psi_d(x, y)
\end{equation}
and
\begin{equation} \label{eq:chebyshev2}
\tilde V_n(x, y) = x^{\frac{1 - (-1)^n}{2}}\prod\limits_{\substack{d \mid n\\1 \leq d < n\\\textrm{$d$ is odd}}}\Psi_{4n/d}(x, y).
\end{equation}
Note that all binary forms in the above factorizations are irreducible. We will prove the following lemma, which implies Theorem \ref{thm:chebyshev}.

\begin{lem} \label{lem:chebyshev2}
Let $n$ be an integer such that $n \geq 3$.

\begin{enumerate}[1.]
\item If $n$ is odd, then
$$
\Aut \tilde V_n = \left\langle
\begin{pmatrix}
1 & 0\\
0 & -1
\end{pmatrix}
\right\rangle
\cong \bm C_2, \quad
\Aut |\tilde V_n|= \left\langle
\begin{pmatrix}
-1 & 0\\
0 & 1
\end{pmatrix},
\begin{pmatrix}
1 & 0\\
0 & -1
\end{pmatrix}
\right\rangle
\cong \bm D_2.
$$

\item If $n$ is even, then
$$
\Aut \tilde V_n = \Aut |\tilde V_n| = \left\langle
\begin{pmatrix}
-1 & 0\\
0 & 1
\end{pmatrix},
\begin{pmatrix}
1 & 0\\
0 & -1
\end{pmatrix}
\right\rangle
\cong \bm D_2.
$$

\item If $n$ is odd, then
$$
\Aut \tilde U_n = \Aut |\tilde U_n| = \left\langle
\begin{pmatrix}
-1 & 0\\
0 & 1
\end{pmatrix},
\begin{pmatrix}
1 & 0\\
0 & -1
\end{pmatrix}
\right\rangle
\cong \bm D_2.
$$

\item If $n$ is even, then
$$
\Aut \tilde U_n = \left\langle
\begin{pmatrix}
1 & 0\\
0 & -1
\end{pmatrix}
\right\rangle
\cong \bm C_2, \quad
\Aut |\tilde U_n|= \left\langle
\begin{pmatrix}
-1 & 0\\
0 & 1
\end{pmatrix},
\begin{pmatrix}
1 & 0\\
0 & -1
\end{pmatrix}
\right\rangle
\cong \bm D_2.
$$
\end{enumerate}
\end{lem}

Let us now see why Theorem \ref{thm:chebyshev} follows from Lemma \ref{lem:chebyshev2}. Note that
$$
\tilde U_n = (U_{n - 1})_S \quad \text{and} \quad \tilde V_n = 2(T_n)_S,
$$
where $S = \left(\begin{smallmatrix}1/2 & 0\\0 & 1\end{smallmatrix}\right)$. Note that $S\left(\begin{smallmatrix}-1 & 0\\0 & 1\end{smallmatrix}\right)S^{-1} = \left(\begin{smallmatrix}-1 & 0\\0 & 1\end{smallmatrix}\right)$ and $S\left(\begin{smallmatrix}1 & 0\\0 & -1\end{smallmatrix}\right)S^{-1} = \left(\begin{smallmatrix}1 & 0\\0 & -1\end{smallmatrix}\right)$. By Lemma \ref{lem:automorphism_of_adjacent_form},
$$
\begin{array}{r l l}
\Aut U_{n - 1} &= S(\Aut \tilde U_n)S^{-1} & = \Aut \tilde U_n,\\
\Aut |U_{n - 1}| & = S(\Aut |\tilde U_n|)S^{-1} & = \Aut |\tilde U_n|,\\
\Aut T_n & = S(\Aut \tilde V_n)S^{-1} & = \Aut \tilde V_n,\\
\Aut |T_n| & = S(\Aut |\tilde V_n|)S^{-1} & = \Aut |\tilde V_n|.
\end{array}
$$
%
%
This concludes the proof of Theorem \ref{thm:chebyshev}. Before we proceed to the proof of Lemma \ref{lem:chebyshev2}, we need to establish one supplementary result.

\begin{lem} \label{lem:chebyshev}
Let $F$ denote either $\tilde U_n$ or $\tilde V_n$, with $\deg F \geq 7$. Suppose that \mbox{$\Psi_k \mid F$} and $\deg \Psi_k \geq 5$. Then $\Aut |F| \subseteq \Aut |\Psi_k|$.

%
%
\end{lem}

\begin{proof}
Let $M = \left(\begin{smallmatrix}s & u\\t & v\end{smallmatrix}\right)$ be an element of $\Aut |F|$. Then
$$
D(F_M) = (\det M)^{\deg F(\deg F-1)}D(F).
$$
Since $F_M = \pm F$, $D(F) \neq 0$ and $\deg F > 1$, we see that $\det M \neq 0$.

Suppose that the binary forms $(\Psi_k)_M$ and $\Psi_k$ are distinct.
Since $\Psi_k$ is irreducible and $\det M \neq 0$, it must be the case that $(\Psi_k)_{M}$ is also irreducible. At this point, we consider two cases.

\emph{Case 1.} Suppose that $(\Psi_k)_{M} = r\Psi_k$ for some non-zero $r \in \mathbb Q$. Then
$$
(\det M)^{\deg \Psi_k(\deg \Psi_k - 1)}D(\Psi_k) = D\left((\Psi_k)_M\right) = D(r\Psi_k) = r^{2(\deg \Psi_k - 1)}D(\Psi_k).
$$
Since $D(\Psi_k) \neq 0$, it must be the case that
$$
(\det M)^{\deg \Psi_k(\deg \Psi_k - 1)} = r^{2(\deg \Psi_k - 1)}.
$$
Since $\det M = \pm 1$, we see that $r^{2(\deg \Psi_k - 1)} = 1$. Since $r \in \mathbb Q$,  we conclude that $r \in \{\pm 1\}$, so $M \in \Aut |\Psi_k|$.

\emph{Case 2.} Suppose that $(\Psi_k)_M$ is not a rational multiple of $\Psi_k$. Since $\Psi_k \mid F$ and the greatest common divisors of the coefficients of $F$ (known as the \emph{content} of $F$) is equal to $1$, there exists some $H(x, y) \in \Z[x, y]$ such that $F = \Psi_kH$. Since $M \in \Aut |F|$, we have
$$
\pm F = F_M = (\Psi_kH)_M = (\Psi_k)_MH_M,
$$
which means that $(\Psi_k)_M \mid F$ in $\Q[x, y]$. Since
\begin{enumerate}[a)]
\item $F$ factors as in (\ref{eq:chebyshev1}) or in (\ref{eq:chebyshev2});
\item both $\Psi_k$ and $(\Psi_k)_M$ are irreducible; and
\item $\deg (\Psi_k)_M = \deg \Psi_k > 1$,
\end{enumerate}
we see that $(\Psi_k)_M = r\Psi_{\ell}$ for some non-zero $r \in \mathbb Q$ and $\ell \in \mathbb N$. Furthermore, since $(\Psi_k)_M$ is not a rational multiple of $\Psi_k$, it must be the case that $k \neq \ell$. Since $(\Psi_\ell)_{M^{-1}} = r^{-1}\Psi_k$, without loss of generality we may assume that $k < \ell$.

Now, since $(\Psi_k)_M = r\Psi_{\ell}$, the polynomials $rm^{\deg \Psi_k}\Psi_\ell(x)$ and \mbox{$\Psi_k(sx + u, tx + v)$} are equal. In particular, their roots are the same, which means that
$$
2\cos\left(\frac{2\pi q}{\ell}\right) = \frac{2\cos\left(\frac{2\pi}{k}\right)v - u}{-2\cos\left(\frac{2\pi}{k}\right)t + s}
$$
for some integer $q$ coprime to $\ell$. Therefore, $2\cos\left(\frac{2\pi q}{\ell}\right) \in \Q\left(2\cos\left(\frac{2\pi}{k}\right)\right)$. By Lemma \ref{lem:galois}, the Galois group of $\mathbb Q\left(2\cos\left(\frac{2\pi}{n}\right)\right)$ is Abelian. Consequently, all the conjugates of $2\cos\left(\frac{2\pi q}{\ell}\right)$, including $2\cos\left(\frac{2\pi}{\ell}\right)$, belong to $\mathbb Q\left(2\cos\left(\frac{2\pi}{k}\right)\right)$, so $\Q\left(2\cos\left(\frac{2\pi}{\ell}\right)\right) \subseteq \Q\left(2\cos(\frac{2\pi}{k})\right)$. Since $M$ is invertible, we conclude that $\Q\left(2\cos\left(\frac{2\pi}{k}\right)\right) = \Q\left(2\cos\left(\frac{2\pi}{\ell}\right)\right)$. Since $k  < \ell$ and $k, \ell \notin \{1, 2, 3, 4, 6\}$, it follows from Lemma \ref{lem:when-fields-are-the-same} that $k$ is odd and $\ell = 2k$. Therefore,
$$
2\cos\left(\frac{\pi q}{k}\right) = \frac{2\cos\left(\frac{2\pi}{k}\right)v - u}{- 2\cos\left(\frac{2\pi}{k}\right)t + s}.
$$
Since $k$ is odd and $q$ is coprime to $2k$, it must be the case that $q$ is odd. Also,
$$
2\cos\left(\frac{\pi q}{k}\right) = -2\cos\left(\pi - \frac{\pi q}{k}\right) = -2\cos\left(\frac{2\pi m}{k}\right),
$$
where $m = \frac{k - q}{2}$ is an integer coprime to $k$. Hence
$$
2\cos\left(\frac{2\pi m}{k}\right) = \frac{-2\cos\left(\frac{2\pi}{k}\right)v + u}{2\cos\left(\frac{2\pi}{k}\right)t - s}.
$$
Since $\deg \Psi_k \geq 5$, it follows from Lemma \ref{lem:stuv} that $s \neq 0$, $s = v$ and \mbox{$t = u = 0$}. Since $s$, $t$, $u$ and $v$ are integers such that $\gcd(s, t, u, v) = 1$, we find that $\gcd(s, v) = 1$. This means that $s = v = \pm 1$ and \mbox{$M \in \{\pm I\} \subseteq \Aut |\Psi_k|$}.
\end{proof}

We will now turn our attention to the proof of the main result of this section.

\begin{proof}[Proof of Lemma \ref{lem:chebyshev2}.]
For $n \geq 7$ consider the binary form $\tilde V_n(x, y)$. Then \mbox{$\Psi_{4n} \mid \tilde V_n$} and $\deg \Psi_{4n} = \varphi(4n)/2 \geq 5$. It follows from Lemma \ref{lem:chebyshev} that $\Aut |\tilde V_n| \subseteq \Aut |\Psi_{4n}|$. By Part 7 of Theorem \ref{thm:cos}, $\Aut |\Psi_{4n}| = \{\pm I, \pm M\}$, where \mbox{$M = \left(\begin{smallmatrix}1 & 0\\0 & -1\end{smallmatrix}\right)$}. We consider two cases.

\emph{Case 1.} If $n$ is even, then there exists a binary form $G(x, y)$ such that $\tilde V_n(x, y) = G(x^2, y^2)$. Then $(\tilde V_n)_A = \tilde V_n$ for any $A \in \Aut |\Psi_{4n}|$. Therefore,
$$
\Aut \tilde V_n = \Aut |\tilde V_n| = \{\pm I, \pm M\}.
$$

\emph{Case 2.} If $n$ is odd, then there exists a binary form $G(x, y)$ such that $\tilde V_n(x, y) = xG(x^2, y^2)$. Then $(\tilde V_n)_{-I} = -\tilde V_n$, $(\tilde V_n)_{M} = \tilde V_n$ and $(\tilde V_n)_{-M} = -\tilde V_n$. Therefore,
$$
\Aut \tilde V_n =\{I, M\}, \quad \Aut |\tilde V_n| =\{\pm I, \pm M\}.
$$

Next, for $n \notin \{4, 5, 6, 7, 8, 9, 10, 12, 15\}$ consider the binary form $\tilde U_n(x, y)$. Then $\Psi_{2n} \mid \tilde U_n$ and $\deg \Psi_{2n} = \varphi(2n)/2 \geq 5$. It follows from Lemma \ref{lem:chebyshev} that $\Aut |\tilde V_n| \subseteq \Aut |\Psi_{2n}|$. By Part 7 of Theorem \ref{thm:cos}, $\Aut |\Psi_{2n}| = \{\pm I, \pm M\}$, where $M = \left(\begin{smallmatrix}1 & 0\\0 & -1\end{smallmatrix}\right)$. We consider two cases.

\emph{Case 1.} If $n$ is even, then there exists a binary form $G(x, y)$ such that $\tilde U_n(x, y) = xG(x^2, y^2)$. Then $(\tilde U_n)_{-I} = -\tilde U_n$, $(\tilde U_n)_{M} = \tilde U_n$ and \mbox{$(\tilde U_n)_{-M} = -\tilde U_n$}. Therefore,
$$
\Aut \tilde U_n =\{I, M\}, \quad \Aut |\tilde U_n| =\{\pm I, \pm M\}.
$$

\emph{Case 2.} If $n$ is odd, then there exists a binary form $G(x, y)$ such that $\tilde U_n(x, y) = G(x^2, y^2)$. Then $(\tilde U_n)_A = \tilde U_n$ for any $A \in \Aut |\Psi_{2n}|$. Therefore,
$$
\Aut \tilde U_n = \Aut |\tilde U_n| = \{\pm I, \pm M\}.
$$

It remains to compute the automorphism groups for thirteen exceptional binary forms whose factors have degree at most $4$:
$$
\tilde V_3, \quad \tilde V_4, \quad \tilde V_5, \quad \tilde V_6, \quad \tilde U_4, \quad \tilde U_5, \quad \tilde U_6, \quad \tilde U_7, \quad \tilde U_8, \quad \tilde U_9, \quad \tilde U_{10}, \quad \tilde U_{12}, \quad \tilde U_{15}.
$$
Notice that $\tilde V_4 = \Psi_{16}$, so the result for this binary form follows from Part 7 of Theorem \ref{thm:cos}. The remaining calculations can be done manually. We will demonstrate them for $\tilde V_3$, $\tilde U_4$ and $\tilde U_{15}$, as the other cases can be established analogously. In what follows we implicitly use the fact that, for any invertible linear fractional transformation $\mu(z) = \frac{vz - u}{-tz + s}$, with $s, t, u, v \in \mathbb Z$, it is the case that $\deg \alpha = \deg \mu(\alpha)$ for any algebraic number $\alpha$.

Consider $\tilde V_3(x, y) = x^3 - 3xy^2$. Let $A = \frac{1}{m}\left(\begin{smallmatrix}s & u\\t & v\end{smallmatrix}\right)$ be an element of $\Aut |\tilde V_3|$, where $s, t, u, v$ and $m = \sqrt{|sv - tu|}$ are integers such that $\gcd(s, t, u, v) = 1$. Then
$$
m^3\tilde V_3(x, y) = \tilde V_3(sx + uy, tx + vy),
$$
which means that the polynomials $m^3\tilde V_3(x, 1)$ and $\tilde V_3(sx + u, tx + v)$ are equal. Since the roots of $\tilde V_3(x, 1)$ are $0$, $\sqrt 3$ and $-\sqrt 3$, it must be the case that either
$$
0 = \frac{v0 - u}{-t0 + s}, \quad \sqrt 3 = \frac{v\sqrt 3 - u}{-t\sqrt 3 + s}, \quad -\sqrt 3 = \frac{v(-\sqrt 3) - u}{-t(-\sqrt 3) + s}
$$
or
$$
0 = \frac{v0 - u}{-t0 + s}, \quad -\sqrt 3 = \frac{v\sqrt 3 - u}{-t\sqrt 3 + s}, \quad \sqrt 3 = \frac{v(-\sqrt 3) - u}{-t(-\sqrt 3) + s}.
$$
In either case from the first equation we find that $u = 0$, and from the second equation we find that $t = 0$. Finally, from the third equation we find that $v/s = \pm 1$. Since $\gcd(s, t, u, v) = 1$, the integers $s$ and $v$ are coprime, so $s = \pm 1$ and $v = \pm 1$. Thus it must be the case that $\Aut |\tilde V_3| = \{\pm I, \pm M\}$, where $M = \left(\begin{smallmatrix}1 & 0\\0 & -1\end{smallmatrix}\right)$. In view of this we also have $\Aut \tilde V_3 = \{I, M\}$. An analogous result can be established for the binary form $\tilde U_4(x, y) = x^3 - 2xy^2$.

Finally, consider
\small
$$
\tilde U_{15}(x, y) = (x - 1)(x + 1)(x^2 - x - 1)(x^2 + x - 1)(x^4 - x^3 - 4x^2 + 4x + 1)(x^4 + x^3 - 4x^2 - 4x + 1).
$$
\normalsize
Let $A = \frac{1}{m}\left(\begin{smallmatrix}s & u\\t & v\end{smallmatrix}\right)$ be an element of $\Aut |\tilde U_{15}|$, where $s, t, u, v$ and $m = \sqrt{|sv - tu|}$ are integers such that $\gcd(s, t, u, v) = 1$. Then
$$
m^{14}\tilde U_{15}(x, y) = \tilde U_{15}(sx + uy, tx + vy),
$$
which means that the polynomials $m^{14}\tilde U_{15}(x, 1)$ and $\tilde U_{15}(sx + u, tx + v)$ are equal. Notice that the rational roots of $\tilde U_{15}(x, 1)$ are given by $\pm 1$, while the roots of \mbox{degree $2$} are given by $\frac{\pm 1 \pm \sqrt 5}{2}$. Thus there are eight possible subcases to consider.
\begin{enumerate}[(i)]
\item $1 = \frac{v1 - u}{-t1 + s}, \quad -1 = \frac{v(-1) - u}{-t(-1) + s}, \quad \frac{1 + \sqrt 5}{2} = \frac{v\left(\frac{1 + \sqrt 5}{2}\right) - u}{-t\left(\frac{1 + \sqrt 5}{2}\right) + s}$.

In this case, we obtain a homogeneous system of $3$ linear equations in $4$ unknowns:
$$
\begin{array}{r l}
s - t + u - v & = 0\\
-s - t + u + v & = 0\\
\frac{1 + \sqrt 5}{2}s + \frac{-3 - \sqrt 5}{2}t + u + \frac{-1 - \sqrt 5}{2}v & = 0
\end{array}
$$
Since $s$, $t$, $u$ and $v$ are integers, this system is equivalent to
$$
\begin{array}{r l}
s - t + u - v & = 0\\
-s - t + u + v & = 0\\
\frac{1}{2}s - \frac{3}{2}t + u - \frac{1}{2}v & = 0\\
\frac{1}{2}s - \frac{1}{2}t - \frac{1}{2}v & = 0
\end{array}
$$
Solving this system yields $s = v$ and $u = v = 0$. Since $\gcd(s, v) = 1$, we conclude that $s = v = \pm 1$. 

\item $-1 = \frac{v1 - u}{-t1 + s}, \quad 1 = \frac{v(-1) - u}{-t(-1) + s}, \quad \frac{1 + \sqrt 5}{2} = \frac{v\left(\frac{1 + \sqrt 5}{2}\right) - u}{-t\left(\frac{1 + \sqrt 5}{2}\right) + s}$.

In this case, $s = t = u = v = 0$, which is impossible.

\item $1 = \frac{v1 - u}{-t1 + s}, \quad -1 = \frac{v(-1) - u}{-t(-1) + s}, \quad \frac{1 - \sqrt 5}{2} = \frac{v\left(\frac{1 + \sqrt 5}{2}\right) - u}{-t\left(\frac{1 + \sqrt 5}{2}\right) + s}$.

In this case, $s = t = u = v = 0$, which is impossible.

\item $-1 = \frac{v1 - u}{-t1 + s}, \quad 1 = \frac{v(-1) - u}{-t(-1) + s}, \quad \frac{1 - \sqrt 5}{2} = \frac{v\left(\frac{1 + \sqrt 5}{2}\right) - u}{-t\left(\frac{1 + \sqrt 5}{2}\right) + s}$.

In this case, $s = v = 0$ and $t = -u$. Since $\gcd(t, u) = 1$, we conclude that $t = \pm 1$ and $u = \mp 1$.

\item $1 = \frac{v1 - u}{-t1 + s}, \quad -1 = \frac{v(-1) - u}{-t(-1) + s}, \quad \frac{-1 + \sqrt 5}{2} = \frac{v\left(\frac{1 + \sqrt 5}{2}\right) - u}{-t\left(\frac{1 + \sqrt 5}{2}\right) + s}$.

In this case, $s = v = 0$ and $t = u$. Since $\gcd(t, u) = 1$, we conclude that $t = u = \pm 1$.

\item $-1 = \frac{v1 - u}{-t1 + s}, \quad 1 = \frac{v(-1) - u}{-t(-1) + s}, \quad \frac{-1 + \sqrt 5}{2} = \frac{v\left(\frac{1 + \sqrt 5}{2}\right) - u}{-t\left(\frac{1 + \sqrt 5}{2}\right) + s}$.

In this case, $s = t = u = v = 0$, which is impossible.

\item  $1 = \frac{v1 - u}{-t1 + s}, \quad -1 = \frac{v(-1) - u}{-t(-1) + s}, \quad \frac{-1 - \sqrt 5}{2} = \frac{v\left(\frac{1 + \sqrt 5}{2}\right) - u}{-t\left(\frac{1 + \sqrt 5}{2}\right) + s}$.

In this case, $s = t = u = v = 0$, which is impossible.

\item $-1 = \frac{v1 - u}{-t1 + s}, \quad 1 = \frac{v(-1) - u}{-t(-1) + s}, \quad \frac{-1 - \sqrt 5}{2} = \frac{v\left(\frac{1 + \sqrt 5}{2}\right) - u}{-t\left(\frac{1 + \sqrt 5}{2}\right) + s}$.

In this case, $s = -v$ and $t = u = 0$. Since $\gcd(s, v) = 1$, we conclude that $s = \pm 1$ and $v = \mp 1$.
\end{enumerate}
We conclude that $A \in \left\{\pm I, \pm M, \pm N, \pm MN\right\}$, where $M = \left(\begin{smallmatrix}1 & 0\\0 & -1\end{smallmatrix}\right)$ and \mbox{$N = \left(\begin{smallmatrix}0 & 1\\-1 & 0\end{smallmatrix}\right)$}. By checking each possible value of $A$ we find that
$$
\Aut \tilde U_{15} = \Aut |\tilde U_{15}| = \{\pm I, \pm M\}.
$$
\end{proof}

\section*{Acknowledgements}

The author is grateful to his PhD advisor, Prof.\ Cameron L.\ Stewart, who proposed to explore this exciting subject, to Prof. \'Etienne Fouvry for correcting the statement of Lemma 3.7 and his other suggestions, as well as to the anonymous reviewer for their excellent advice on how to improve the article.

\end{document}